\documentclass[10pt,leqno]{amsart}
\usepackage{amsmath}
\usepackage{amsfonts}
\usepackage{amssymb}
\usepackage{mathtools}
\usepackage{graphicx}
\usepackage{color}
\usepackage{hyperref}
\usepackage{xcolor}
\newtheorem{theorem}{Theorem}[section]
\newtheorem*{theorem*}{Theorem}
\newtheorem{lemma}[theorem]{Lemma}
\newtheorem{corollary}[theorem]{Corollary}
\newtheorem{proposition}[theorem]{Proposition}

\theoremstyle{definition}
\newtheorem{definition}[theorem]{Definition}
\newtheorem{example}[theorem]{Example}

\newtheorem{remark}{Remark}[section]
\newtheorem{question}{Question}[section]

\newcommand{\R}{\mathbb{R}}

\def \b {\beta}

\def\Ric{\text{Ric}}

\def\a{\alpha}
\def\l{\lambda}

\def\R{\mathbb{R}}

\def\S{{\mathbb{S}}}
\def\CP{{\mathbb{CP}}}

\def\vp{\varphi}

\def\k{\kappa}
\def\id{\operatorname{id}}
\def\Ric{\operatorname{Ric}}

\def\tr{\operatorname{tr}}

\def\spn{\operatorname{span}}

\numberwithin{equation}{section}
\makeatletter
\newcommand*\owedge{\mathpalette\@owedge\relax}
\newcommand*\@owedge[1]{%
  \mathbin{%
    \ooalign{%
      $#1\m@th\bigcirc$\cr
      \hidewidth$#1\m@th\wedge$\hidewidth\cr
    }%
  }%
}
\makeatother

\begin{document}

\title[Curvature Operator of the Second Kind]{New Sphere Theorems under Curvature Operator of the Second Kind}

\author[Xiaolong Li]{Xiaolong Li}\thanks{The author's research is partially supported by NSF-DMS \#2405257, NSF LEAPS-MPS \#2316659, and Simons Collaboration Grant \#962228}
\address{Department of Mathematics and Statistics, Auburn University, Auburn, AL, 36849, USA}
\email{xil0005@auburn.edu}

\subjclass[2020]{53C20, 53C21, 53C24, 53C55}

\keywords{Sphere theorems, Curvature operator of the second kind, Curvature pinching, Betti numbers, Positive isotropic curvature}

\begin{abstract}
We investigate Riemannian manifolds $(M^n,g)$ whose curvature operator of the second kind $\mathring{R}$ satisfies the condition
\begin{equation*}
\alpha^{-1} (\lambda_1 +\cdots +\lambda_{\alpha}) >  - \theta \bar{\lambda}, 
\end{equation*}
where $\lambda_1  \leq \cdots \leq \lambda_{(n-1)(n+2)/2}$ are the eigenvalues of $\mathring{R}$, $\bar{\lambda}$ is their average, and $\theta > -1$. 
Under such curvature conditions with optimal $\theta$ (depending on $n$ and $\alpha$), we prove differentiable sphere theorems and their rigidity results in dimensions three and four, a homological sphere theorem in higher dimensions, and a curvature characterization of K\"ahler space forms. These results generalize recent works corresponding to $\theta =0$ (namely $\a$-positivity of $\mathring{R}$) of Cao-Gursky-Tran, Nienhaus-Petersen-Wink, and the author. Moreover, examples are provided to demonstrate the sharpness of most results. 
\end{abstract}
\maketitle

\section{Introduction}

A central theme in geometry is to understand how curvature conditions determine the topology of the underlying space. Of great importance are the sphere theorems, which state that the underlying manifolds must be (up to homology, homeomorphism, or diffeomorphism) spherical space forms under suitable curvature conditions. 
For instance, the famous differentiable sphere theorem due to Brendle and Schoen \cite{BS09} asserts that a closed Riemannian manifold with strictly quarter-pinched sectional curvature is diffeomorphic to a spherical space form. Much earlier, Berger \cite{Berger60} and Klingenberg \cite{Klingenberg61} proved this result on the homeomorphism level. Another celebrated differentiable sphere theorem, proved using the Ricci flow by Hamilton \cite{Hamilton82} in dimension three, Hamilton \cite{Hamilton86} and Chen \cite{Chen91} in dimension four, and B\"ohm and Wilking \cite{BW08} in all higher dimensions, states that a closed Riemannian manifold with two-positive curvature operator is diffeomorphic to a spherical space form. On the homology level, this was proved much earlier by Meyer \cite{Meyer71} using the Bochner technique. We refer the reader to \cite{BS11}, \cite{Brendle10book}, \cite{NW10}, \cite{PW21}, and the references therein for more sphere theorems, their corresponding rigidity results, and further developments.

Recently, a new differentiable sphere theorem was proved under the condition of three-positive curvature operator of the second kind. More generally, it has been shown that
\begin{theorem}[\cite{CGT23}, \cite{Li21}, and \cite{NPW22}]\label{thm Nishikawa}
A closed Riemannian manifold with three-nonnegative curvature operator of the second kind is either flat or diffeomorphic to a spherical space form. 
\end{theorem}

Theorem \ref{thm Nishikawa} settles a conjecture of Nishikawa \cite{Nishikawa86} under weaker assumptions but with stronger conclusions. The original conjecture states that a closed Riemannian manifold with positive (respectively, nonnegative) curvature operator of the second kind is diffeomorphic to a spherical space form (respectively, a Riemannian locally symmetric space). The positive part was resolved in a pioneering paper by Cao, Gursky, and Tran \cite{CGT23}. They found that two-positive curvature operator of the second kind implies the PIC1 condition (i.e. $M\times \R$ has positive isotropic curvature) and then appealed to Brendle's convergence result \cite{Brendle08} of closed Ricci flows with PIC1 to constant sectional curvature. Shortly after, the author \cite{Li21} weakened their assumption to three-positive curvature operator of the second kind and also classified closed manifolds with three-nonnegative curvature operator of the second kind: they are either flat, or diffeomorphic to spherical space forms, or isometric to compact irreducible symmetric spaces. A few months later, Nienhaus, Petersen, and Wink \cite{NPW22} discovered a new Bochner formula for the curvature operator of the second kind and used it to prove that a closed Riemannian manifold with $\frac{n+2}{2}$-nonnegative curvature operator of the second kind must be either flat or a rational homology sphere, thus ruling out compact irreducible symmetric spaces in the author's classification and yielding Theorem \ref{thm Nishikawa}.

In the above discussion, the curvature operator (of the first kind by Nishikawa's terminology \cite{Nishikawa86}) $\hat{R}: \wedge^2(TM) \to \wedge^2(TM)$ refers to the action of the Riemann curvature tensor $R_{ijkl}$ on two-forms via
    \begin{equation*}
        \hat{R}(\omega)_{ij}=\frac{1}{2}\sum_{k,l=1}^n R_{ijkl}\omega_{kl}.
    \end{equation*}
Throughout the paper, we use the convention that $R_{ijij} >0$ on the round sphere. 
Understanding the geometric and topological consequences of positivity conditions on the curvature operator is of longstanding interest in Riemannian geometry; see \cite{Meyer71}, \cite{Tachibana74}, \cite{GM75}, \cite{Hamilton82, Hamilton86}, \cite{Chen91}, \cite{BW08}, \cite{NW07} and \cite{PW21}, etc. 
The curvature operator of the second kind $\mathring{R}$ is defined by
$$\mathring{R} = \pi \circ \overline{R}: S^2_0(TM) \to  S^2_0(TM),$$
where $\overline{R}:S^2(TM) \to S^2(TM)$ is the action of $R_{ijkl}$ on symmetric two-tensors via
\begin{equation*}
    \overline{R}(h)_{ij} =\sum_{k,l=1}^n R_{iklj}h_{kl},
\end{equation*}
and $\pi: S^2(TM) \to S^2_0(TM)$ is the projection map from symmetric two-tensors to traceless symmetric two-tensors. For a real number $\a \in [1,(n-1)(n+2)/2]$, a Riemannian manifold $(M^n,g)$ is said to have $\a$-nonnegative curvature operator of the second kind if for any $p\in M$, 
\begin{equation*}
     \l_1 + \cdots +\l_{\lfloor{\a}\rfloor} +(\a-\lfloor{\a}\rfloor)\l_{\lfloor{\a}\rfloor+1} \geq 0, 
\end{equation*}
where $\l_1 \leq \cdots \leq \l_{(n-1)(n+2)/2}$ are the eigenvalues of $\mathring{R}$ at $p$ and 
$$\lfloor{x}\rfloor:=\max\{ m \in \mathbb{Z}: m \leq x\}$$ denotes the floor function. Similarly, one defines $\a$-positivity, $\a$-negativity, and $\a$-nonpositivity of any symmetric operator.

The action of the Riemann curvature tensor on symmetric two-tensors indeed has a long history. It appeared for K\"ahler manifolds in the study of the deformation of complex analytic structures by Calabi and Vesentini \cite{CV60}. They introduced the self-adjoint operator $\xi_{\a \b} \to R_{\rho\a\b\sigma} \xi_{\rho \sigma}$ from $S^2(T^{1,0}_p M)$ to itself, and computed the eigenvalues of this operator on Hermitian symmetric spaces of classical type, with the exceptional ones handled shortly after by Borel \cite{Borel60}. In the Riemannian setting, the operator $\overline{R}$ arises naturally in the context of deformations of Einstein structures in Berger and Ebin \cite{BE69} (see also \cite{Koiso79a, Koiso79b} and \cite{Besse08}). 
In addition, it appears in the Bochner-Weitzenb\"ock formulas for symmetric two-tensors (see for example \cite{MRS20}), for differential forms in \cite{OT79}, and for the Riemann curvature tensor in \cite{Kashiwada93}.  
In another direction, curvature pinching estimates for $\overline{R}$ were studied by Bourguignon and Karcher \cite{BK78}, and they calculated eigenvalues of $\overline{R}$ on the complex projective space with the Fubini-Study metric and the quaternionic projective space with its canonical metric. 
Nevertheless, the operators $\overline{R}$ and $\mathring{R}$ are significantly less investigated than $\hat{R}$ and it is our goal to achieve a better understanding of them. 

The resolution of Nishikawa's conjecture has triggered a series of works investigating the curvature operator of the second kind, including \cite{Li22JGA, Li22PAMS, Li22Kahler, Li22product}, \cite{NPWW22}, \cite{FL24}, \cite{DF24}, and \cite{DFY24}. Most of them try to understand the geometric and topological implications of $\a$-nonnegative curvature operator of the second kind and prove improved results by increasing $\a$ (or equivalently weakening the curvature condition). For example, the author \cite[Theorem 1.4]{Li22product} obtained a classification (up to homeomorphism) of closed Riemannian manifolds with $4\frac{1}{2}$-nonnegative curvature operator of the second kind, generalizing Theorem \ref{thm Nishikawa}.  

The purpose of this paper is to introduce new lower bound conditions on the curvature operator of the second kind (see Definition \ref{def cone}) and prove optimal sphere theorems that extend several above-mentioned results. 

Let $(V,g)$ be a (real) Euclidean vector space of dimension $n\geq 3$ and denote by $S^2_0(V)$ the space of traceless symmetric two-tensors on $V$. Throughout this paper, we write
\begin{equation*}
    N=N(n):=\dim(S^2_0(V))=\tfrac{(n-1)(n+2)}{2},
\end{equation*}
and we use the convention
\begin{equation*}
\l_1+\cdots +\l_{\a}:=\l_1 + \cdots +\l_{\lfloor{\a}\rfloor} +(\a-\lfloor{\a}\rfloor)\l_{\lfloor{\a}\rfloor+1}
\end{equation*}
when $\a \in [1,N]$ is not an integer. 

\begin{definition}\label{def cone}
Let $\a \in [1,N)$ and $\theta >-1$. 
\begin{enumerate}
    \item We define $\mathcal{C}(\a,\theta)$ to be the cone of symmetric operators $\mathring{R} : S^2_0(V) \to S^2_0(V)$ satisfying 
\begin{equation}\label{eq C2K}
\a^{-1} \left(\l_1 + \cdots +\l_{\a} \right) \geq -\theta \bar{\l}, 
\end{equation}
where $\l_1\leq   \cdots  \leq \l_N$ are the eigenvalues of $\mathring{R}$ and $\bar{\l}$ denotes their average.  
\item We denote by $\mathring{\mathcal{C}}(\a,\theta)$ and $\partial \mathcal{C}(\a,\theta)$ the interior and the boundary of $\mathcal{C}(\a,\theta)$, respectively. 
\item We say a Riemannian manifold $(M^n,g)$ satisfies $\mathring{R}\in \mathcal{C}(\a,\theta)$ (respectively, $\mathring{R}\in \mathring{\mathcal{C}}(\a,\theta)$) if $\mathring{R}_p \in \mathcal{C}(\a,\theta)$ (respectively, $\mathring{R}_p \in \mathring{\mathcal{C}}(\a,\theta)$) for all $p\in M$, where $\mathring{R}_p$ denotes the curvature operator of the second kind at $p$. 
\end{enumerate}
\end{definition}

Note that $\mathring{R} \in \mathcal{C}(\a, 0)$ if and only if $\mathring{R}$ is $\a$-nonnegative and $\mathring{R} \in \mathcal{C}(1,\theta)$ if and only if $\mathring{R}+\theta \bar{\l} \id$ is nonnegative. For general $\a$ and $\theta$, $\mathring{R} \in \mathcal{C}(\a, \theta)$ can be interpreted as that the average of the smallest $\a$ eigenvalues of $\mathring{R}$ is bounded from below by $-\theta \bar{\l}$. Thus, the conditions $\mathring{R} \in \mathcal{C}(\a,\theta)$ give a two-parameter family of lower bounds on $\mathring{R}$.  

Our main motivation to introduce the conditions $\mathring{R} \in \mathcal{C}(\a,\theta)$ comes from proving optimal differentiable sphere theorems. Theorem \ref{thm Nishikawa} implies that the sum of the smallest three eigenvalues of $\mathring{R}$ is indeed negative on all compact symmetric spaces (with their canonical metrics), except spherical space forms. This suggests that a closed Riemannian manifold satisfying $\mathring{R} \in \mathring{\mathcal{C}}(3,\theta)$ with $\theta >0$ sufficiently small should be diffeomorphic to a spherical space form. More ambitiously, one can ask: 
\begin{question}
Given $n\geq 3$ and $\a \in [1,N)$, what is the largest number $\bar{\theta}(n, \alpha)$ such that a closed Riemannian manifold satisfying $\mathring{R} \in \mathring{\mathcal{C}}(\a,\bar{\theta}(n, \alpha))$ is diffeomorphic to a spherical space form? 
\end{question} 

In this paper, we completely answer this question in dimensions three and four and provide a partial result in higher dimensions. 
Note that (see Example \ref{example cylinder}) the curvature operator of the second kind of $\mathbb{S}^{n-1}\times \mathbb{S}^1$ (with the standard product metric) lies on $\partial \mathcal{C}(\a,\bar{\Theta}_{n,\a} )$, where
\begin{equation}\label{eq theta cylinder}
\bar{\Theta}_{n,\a}:= \begin{cases}
        \a^{-1}, & 1\leq \a \leq n, \\
        \a^{-1}+\frac{n(n-\a)}{(n-2)\a}, & n\leq \a <N.
        \end{cases}
\end{equation}
Therefore, we must have $\bar{\theta}(n, \alpha) \leq \bar{\Theta}_{n,\a}$. Below we shall show that $\bar{\theta}(n,\a) =  \bar{\Theta}_{n,\a}$ for $n=3$ and $n=4$.

In dimension three, we prove that
\begin{theorem}\label{thm 3D}
Let $(M^3,g)$ be a closed Riemannian manifold of dimension three. Let $1\leq \a <5=N(3)$ and $\bar{\Theta}_{3,\a}$ be defined as in \eqref{eq theta cylinder}. 
\begin{enumerate}
\item If $(M,g)$ satisfies $\mathring{R} \in \mathring{\mathcal{C}}(\a, \bar{\Theta}_{3,\a})$, then $M$ is diffeomorphic to a spherical space form. 
\item If $(M,g)$ satisfies $\mathring{R} \in \mathcal{C}(\a, \bar{\Theta}_{3,\a})$, then $M$ is diffeomorphic to a quotient of one of the spaces $\mathbb{S}^3$, or $\mathbb{S}^2 \times \R$, or $\R^3$ by a group of fixed point free isometries in the standard metrics. 
\end{enumerate}
\end{theorem}

Other than the case $\a=3\frac{1}{3}$ which was proved by the author \cite[Theorem 1.7]{Li22JGA}, Theorem \ref{thm 3D} is new for all other $\a \in [1,5)$. The key is to establish implications of $\mathring{R} \in \mathcal{C}(\a,\bar{\Theta}_{n,\a})$ on the Ricci curvature. More precisely, we prove the following result in all dimensions. 
\begin{proposition}\label{prop Ricci nonnegative}
Let $\a \in [1,N)$ and $\bar{\Theta}_{n,\a}$ be defined as in \eqref{eq theta cylinder}. Let $R\in S^2_B(\wedge^2 V)$ be an algebraic curvature operator and $\mathring{R}$ its induced curvature operator of the second kind. 
If $\mathring{R} \in \mathcal{C}(\a, \bar{\Theta}_{n,\a})$ (respectively, $\mathring{R} \in \mathring{\mathcal{C}}(\a, \bar{\Theta}_{n,\a})$), then $R$ has nonnegative (respectively, positive) Ricci curvature. 
\end{proposition}

Proposition \ref{prop Ricci nonnegative} is optimal on $\mathbb{S}^{n-1} \times \mathbb{S}^1$. Indeed, the proof of Proposition \ref{prop Ricci nonnegative} (and the more general Proposition \ref{prop Ricci}) uses $\mathbb{S}^{n-1} \times \mathbb{S}^1$ as a model space by applying $\mathring{R}$ to the eigentensors of the curvature operator of the second kind on $\mathbb{S}^{n-1} \times \mathbb{S}^1$. This strategy has been successfully employed by the author in previous works \cite{Li22JGA, Li22PAMS, Li22Kahler, Li22product} with model spaces such as $\CP^m$, $\S^k \times \S^{n-k}$, and $\CP^k \times \CP^{m-k}$. 
With Proposition \ref{prop Ricci nonnegative}, Theorem \ref{thm 3D} then follows from Hamilton's famous classification of closed three-manifolds with positive/nonnegative Ricci curvature in \cite{Hamilton82, Hamilton86}. Alternatively, Proposition \ref{prop Ricci nonnegative} can be proved using the explicit expressions for the eigenvalues of $\mathring{R}$ in terms of that of $\hat{R}$ in dimension three found by Fluck and the author in \cite{FL24}.

In dimension four, we prove that 
\begin{theorem}\label{thm 4D}
Let $(M^4,g)$ be a closed Riemannian manifold of dimension four. Let $1\leq \a <9=N(4)$ and $\bar{\Theta}_{4,\a}$ be defined as in \eqref{eq theta cylinder}.
\begin{enumerate}
    \item If $(M,g)$ satisfies $\mathring{R} \in \mathring{\mathcal{C}}(\a, \bar{\Theta}_{4,\a})$, then $M$ is diffeomorphic to $\mathbb{S}^4$ or $\mathbb{RP}^4$. 
    \item If $(M,g)$ satisfies $\mathring{R} \in \mathcal{C}(\a, \bar{\Theta}_{4,\a})$, then one of the following statements holds:
    \begin{itemize}
        \item [(a)] $(M,g)$ is flat;
        \item [(b)] $M$ is diffeomorphic to $\mathbb{S}^4$ or $\mathbb{RP}^4$;
        \item [(c)] $1\leq \a \leq 4$ and the universal cover of $(M,g)$ is diffeomorphic to $\mathbb{S}^3 \times \R$;
        \item [(d)] $4<\a <9$ and the universal cover of $(M,g)$ is isometric to $\mathbb{S}^3 \times \R$;
        \item [(e)]$4\leq \a <9$ and $(M,g)$ is isometric to $\mathbb{CP}^2$ with the Fubini-Study metric. 
    \end{itemize}
\end{enumerate}    
\end{theorem}

Previously, Theorem \ref{thm 4D} was only known for $\a=4\frac{1}{2}$ by \cite[Theorem 1.4]{Li22product}. We point out that all the cases in part (2) of Theorem \ref{thm 4D} can occur. The diffeomorphism in (2c) cannot be upgraded to an isometry, as $N^3 \times \R$ satisfies $\mathring{R} \in \mathcal{C}(\a, \bar{\Theta}_{4,\a})$ for any $1\leq \a \leq 4$ as long as $N^3$ has positive curvature operator of the second kind (see \cite[Proposition 2.1]{Li22product}). For (2e), we remark that $\mathbb{CP}^2$ satisfies $\mathring{R} \in \mathcal{C}(\a, \bar{\Theta}_{4,\a})$ if and only if $4\leq \a <9$ (see Example \ref{example complex projective space}). 

To prove Theorem \ref{thm 4D}, we derive, as in \cite{Li22JGA}, implications of $\mathring{R} \in \mathcal{C}(\a,\theta)$ on the isotropic curvature, a notion that played a central role in the proof of the quarter-pinched differentiable sphere theorem in \cite{BS08}. 
In Proposition \ref{prop PIC}, we show using $\mathbb{CP}^2$ as a model space that in dimension four positive isotropic curvature is implied by a slightly weaker condition than $\mathring{R} \in \mathring{\mathcal{C}}(\a, \bar{\Theta}_{4,\a})$. Hence, $\mathring{R} \in \mathring{\mathcal{C}}(\a, \bar{\Theta}_{4,\a})$ implies both positive Ricci curvature and positive isotropic curvature in dimension four. Part (1) of Theorem \ref{thm 4D} then follows from Hamilton's work \cite{Hamilton97}. The proof of part (2) requires further investigation using \cite{Li22product} when $M$ is locally reducible and also uses the $m=2$ case of Theorem \ref{thm Kahler flat}.

Theorem \ref{thm 3D} and Theorem \ref{thm 4D} imply $\bar{\theta}(n,\a)=\bar{\Theta}_{n,\a}$ for $n=3$ and $n=4$, respectively. One may wonder whether $\bar{\theta}(n,\a)=\bar{\Theta}_{n,\a}$ remains true for any $n\geq 5$. This speculation is supported for $\a = \frac{n+2}{2}$ by the following homological sphere theorem in higher dimensions.

\begin{theorem}\label{thm homology sphere}
Let $(M^n,g)$ be a closed Riemannian manifold of dimension $n\geq 5$. Suppose $(M,g)$ satisfies $\mathring{R} \in \mathcal{C}\left(\frac{n+2}{2}, \theta \right)$ for some $-1 < \theta < \frac{2}{n+2}$. Then $(M,g)$ is either flat or a rational homology sphere.  
\end{theorem}

Taking $\theta =0$ in Theorem \ref{thm homology sphere} recovers the homological sphere theorem of Nienhaus, Petersen, and Wink \cite[Theorem A]{NPW22}. 
The condition $\theta < \frac{2}{n+2}$ is optimal, as $\mathbb{S}^{n-1}\times \mathbb{S}^1$ satisfies $\mathring{R} \in \partial \mathcal{C}(\frac{n+2}{2}, \frac{2}{n+2})$. 
To prove Theorem \ref{thm homology sphere}, we make use of the Bochner formula (see \eqref{eq Bochner} and \eqref{eq Bochner curvature term}) derived in \cite{NPW22}. 
Together with a weight principle (see \cite[Theorem 3.6]{NPW22}), they also proved the vanishing of the $p$-th Betti number under $C(n,p)$-positivity of $\mathring{R}$, where $C(n,p)$ is an explicit constant. Using their method, we prove the following more general result, which implies Theorem \ref{thm homology sphere}. 
\begin{theorem}\label{thm Betti}
Let $(M^n,g)$ be a closed Riemannian manifold of dimension $n\geq 5$ and $2\leq p \leq \frac{n}{2}$. Set
\begin{equation}\label{eq A_{n,p} def}
A_{n,p}:= \frac{2(n-1)(np+n-p^2)}{2(n-1)(n-2p)(n-p+1)+(n-p)(n+2)(n-p+2)}.
\end{equation}
\begin{enumerate}
\item If $(M,g)$ satisfies $\mathring{R} \in \mathring{\mathcal{C}}\left(\frac{n+2}{2}, A_{n,p}\right)$, then the $p$-th Betti number $b_p(M,\R)$ vanishes.  
\item If $(M,g)$ satisfies $\mathring{R} \in \mathcal{C}\left(\frac{n+2}{2}, \theta \right)$ for some $\theta < A_{n,p}$, then either $b_p(M,\R)$ vanishes or $(M,g)$ is flat.
\item If $(M,g)$ satisfies $\mathring{R} \in \mathcal{C}\left(\frac{n+2}{2}, A_{n,p}\right)$, then all harmonic $p$-forms are parallel. 
\end{enumerate}
\end{theorem}

Note that $A_{n,p}$ increases as $p\leq \frac{n}{2}$ increases and the weakest curvature condition occurs when $p=\frac{n}{2}$ with $A_{n,\frac{n}{2}}=\frac{2(n-1)}{n+2}$.  We show that this weakest condition is sufficient for Einstein manifolds. 
\begin{theorem}\label{thm homology sphere Einstein}
Let $(M^n,g)$ be a closed Einstein manifold of dimension $n\geq 4$. 
\begin{enumerate}
\item If $(M,g)$ satisfies $\mathring{R} \in \mathring{\mathcal{C}}\left(\frac{n+2}{2}, \frac{2(n-1)}{(n+2)}\right)$, then $M$ is a rational homology sphere.  
\item If $(M,g)$ satisfies $\mathring{R} \in  \mathcal{C}\left(\frac{n+2}{2}, \theta \right)$ for some $-1 < \theta < \frac{2(n-1)}{(n+2)}$, then $M$ is either flat or a rational homology sphere.
\item If $(M,g)$ satisfies $\mathring{R} \in  \mathcal{C}\left(\frac{n+2}{2}, \frac{2(n-1)}{(n+2)}\right)$, then all harmonic $p$-forms are parallel.
\end{enumerate}
\end{theorem}

We point out that the number $\frac{2(n-1)}{(n+2)}$ in Theorem \ref{thm homology sphere Einstein} is the best possible, as both $\mathbb{CP}^{\frac{n}{2}}$ (with the Fubini-Study metric) and $\mathbb{S}^{k}\times \mathbb{S}^{n-k}$ (with $2\le k\leq \frac{n}{2}$ and the product metric being Einstein) satisfy $\mathring{R} \in \partial \mathcal{C}(\frac{n+2}{2}, \frac{2(n-1)}{n+2})$. See Example \ref{example complex projective space} and Example \ref{example SXS}.

It is also interesting to study K\"ahler manifolds satisfying the condition $\mathring{R} \in \mathcal{C}(\a,\theta)$. By \cite{BK78} (see also Example \ref{example complex projective space}), $(\mathbb{CP}^m, g_{FS})$ satisfies $\partial \mathcal{C}(\a,B_{m,\a})$, where
\begin{equation}\label{eq def Xi_{m,a}}
    B_{m,\a} := \begin{cases}
        \frac{2m-1}{m+1}, & 1\leq \a \leq m^2-1;\\
        \frac{2m-1}{m+1}\frac{3(m^2-1)-2\a}{\a}, & m^2-1 \leq \a < (2m-1)(m+1).
    \end{cases}
\end{equation}
It has been shown (see \cite[Theorem 1.9]{Li21}, \cite[Theorem C]{NPWW22}, and \cite[Theorem 1.2]{Li22Kahler}) that a K\"ahler manifold of complex dimension $m\geq 2$ satisfying either $\mathring{R} \in \mathcal{C}(\a,0)$ or $-\mathring{R} \in \mathcal{C}(\a,0)$ is flat if $1\leq \a < \frac{3}{2}(m^2-1)$ and has constant holomorphic sectional curvature if $\a=\frac{3}{2}(m^2-1)$. Here we prove the following optimal extension to all $\a \in [1, (2m-1)(m+1))$, except $\a =m^2-1$.   
\begin{theorem}\label{thm Kahler flat}
Let $(M^m,g)$ be a K\"ahler manifold of complex dimension $m\geq 2$. Let $1 \leq \a  < (2m-1)(m+1)$ and $B_{m,\a}$ be defined as in \eqref{eq def Xi_{m,a}}. 
\begin{enumerate}
    \item If $\a \neq m^2-1$ and $(M,g)$ satisfies $\mathring{R} \in \mathcal{C}(\a, B_{m,\a})$ (respectively $-\mathring{R} \in \mathcal{C}(\a, B_{m,\a})$), then $(M,g)$ has constant nonnegative (respectively, nonpositive) holomorphic sectional curvature. 
    \item If $(M,g)$ satisfies either $\mathring{R} \in \mathcal{C}(\a, \theta)$ or $-\mathring{R} \in \mathcal{C}(\a, \theta)$ for some $\theta < B_{m,\a}$, then $(M,g)$ is flat. 
\end{enumerate}
\end{theorem}

We point out that part (1) of Theorem \ref{thm Kahler flat} fails for $\a=m^2-1$, as $\mathbb{CP}^{k}\times \mathbb{CP}^{m-k}$ satisfied $\mathring{R} \in \partial \mathcal{C}(m^2-1, \frac{2m-1}{m+1})$ (see Example \ref{example CPXCP}). The case of K\"ahler surfaces (i.e., $m=2$) can be alternatively proved using the normal form of the curvature operator of the second kind in real dimension four discovered in \cite{CGT23}, in a similar way as in \cite{Li22PAMS}.

Given that the Ricci flow is the most powerful tool in proving differentiable sphere theorems and thus is a possible approach to Question A, we conclude this section with the following question.

\begin{question}
For what values of $\a \in [1,N)$ and $\theta >-1$ does the Ricci flow (on closed manifolds) preserve the curvature condition $\mathring{R} \in \mathcal{C}(\a,\theta)$? 
\end{question}

Fluck and the author \cite[Proposition 5.3]{FL24} proved that three-dimensional Ricci flows on closed manifolds preserve the condition $\mathring{R} \in \mathcal{C}(\a,\theta)$ for all $\a \in [1,5)$ and $\theta > -1$. It remains an interesting question in higher dimensions $n\geq 4$. 

\section{Preliminaries}

\subsection{Notation and Conventions}

Let $(V,g)$ be a real Euclidean vector space of dimension $n \geq 3$ and $\{e_i\}_{i=1}^n$ be an orthonormal basis of $V$. We always identify $V$ with its dual space $V^*$ via the inner product $g$. 

Denote by $\wedge^2(V)$, $S^2(V)$, and $S^2_0(V)$ the spaces of two-forms, symmetric two-tensors, and traceless symmetric two-tensors on $V$, respectively. Note that $S^2(V)$ splits into $O(V)$-irreducible subspaces as
    \begin{equation*}
        S^2(V)=S^2_0(V)\oplus \R g. 
    \end{equation*}
The tensor product $\otimes$ is defined by 
\begin{equation*}
        (e_i\otimes e_j)(e_k,e_l)=\delta_{ik}\delta_{jl}.
    \end{equation*}
The symmetric product $\odot$ and the wedge product $\wedge$ are defined by 
 \begin{equation*}
        e_i \odot e_j=e_i\otimes e_j +e_j \otimes e_i,
    \end{equation*}
    and 
\begin{equation*}
        e_i \wedge e_j=e_i\otimes e_j - e_j \otimes e_i,
    \end{equation*}
respectively. 
The inner product on $\wedge^2(V)$ is given by 
    \begin{equation*}
        \langle A, B \rangle =\frac{1}{2}\tr(A^T B),
    \end{equation*}
so $\{e_i \wedge e_j\}_{1\leq i<j\leq n}$ is an orthonormal basis of $\wedge^2(V)$. 
The inner product on $S^2(V)$ is given by 
    \begin{equation*}
        \langle A, B \rangle =\tr(A^T B),
    \end{equation*}
so $\{\frac{1}{\sqrt{2}}e_i \odot e_j\}_{1\leq i<j\leq n} \cup \{\frac{1}{2}e_i \odot e_i\}_{1\leq i\leq n}$ is an orthonormal basis of $S^2(V)$. 

$S^2(\wedge^2 V)$, the space of symmetric two-tensors on $\wedge^2(V)$, has the orthogonal decomposition 
\begin{equation*}
    S^2(\wedge^2 V) =S^2_B(\wedge^2 V) \oplus \wedge^4 V,
\end{equation*}
where $S^2_B(\wedge^2 V)$ consists of all tensors $R\in S^2(\wedge^2(V))$ that also satisfy the first Bianchi identity. $S^2_B(\wedge^2 V)$ is called the space of algebraic curvature operators (or tensors). 

\subsection{Curvature Operator of the Second Kind}
Given $R\in S^2_B(\wedge^2(V))$, the induced symmetric operator $\hat{R}:\wedge^2 (V) \to \wedge^2(V)$ given by 
    \begin{equation*}
        \hat{R}(\omega)_{ij}=\frac{1}{2}\sum_{k,l=1}^n R_{ijkl}\omega_{kl}, 
    \end{equation*}
is called the curvature operator (or the curvature operator of the first kind by Nishikawa \cite{Nishikawa86}).

By the symmetries of $R\in S^2_B(\wedge^2(V))$, $R$ also induces a symmetric operator $\overline{R}:S^2(V) \to S^2(V)$ via 
\begin{equation*}
    \overline{R}(\vp)_{ij}=\sum_{k,l=1}^n R_{iklj}\vp_{kl}. 
\end{equation*}
However, the nonnegativity of $\overline{R}$ is too strong in the sense that $\overline{R}$ is nonnegative if and only if $\overline{R}=0$. 
The curvature operator of the second kind, following Nishikawa's terminology \cite{Nishikawa86}, refers to the symmetric operator 
$$\mathring{R} = \pi \circ \overline{R}: S^2_0(V) \to  S^2_0(V),$$
where $\pi: S^2(V) \to S^2_0(V)$ is the projection map.

We collect some known properties of $\mathring{R}$. 

\begin{proposition}\label{prop basic C2K}
Let $R \in S^2_B(\wedge^2 V)$ and $\mathring{R}$ be its induced curvature operator of the second kind. Then
\begin{enumerate}
    \item $\tr(\mathring{R}) = \frac{n+2}{2n}S$, where $S$ denotes the scalar curvature of $R$.
    \item $\mathring{R}=\id_{S^2_0(V)}$ if $R$ has constant sectional curvature $1$.
    \item $\mathring{R}(\vp,\psi) =\overline{R}(\vp,\psi)$ for $\vp, \psi \in S^2_0(V)$. 
    \item If $\mathring{R}$ is two-positive, then $R$ has positive sectional curvatures.   
    \item If $\mathring{R}$ is positive, then $R$ has positive complex sectional curvature. 
    \item If $n\geq 4$ and $\mathring{R}$ is $3$-positive, then the expression 
    $$R_{1313}+\l^2 R_{1414}+R_{2323}+\l^2R_{2424} -2\l R_{1234}$$ is positive for all orthonormal four frame $\{e_1,e_2,e_3,e_4\}$ and all $\l \in [-1,1]$. 
    \item If $n\geq 4$ and $\mathring{R}$ is $4\frac{1}{4}$-positive, then $R$ has positive isotropic curvature. 
    \item If $\mathring{R}$ is $\left(n+\frac{n-2}{n}\right)$-positive, then $R$ has positive Ricci curvature.
\end{enumerate}
Moreover, the statements in (4)-(8) remain true if ``positive" is replaced by ``nonnegative", or ``nonpositive", or ``negative". 
\end{proposition}
\begin{proof}
(1) and (2) are well-known. See \cite{Li21} or \cite{NPW22}. \\
(3). This says the symmetric bilinear form induced by $\mathring{R}$ is the same as the restriction to $S^2_0(V)$ of the symmetric bilinear form induced by $\overline{R}$. It can be seen as
$$\mathring{R}(\vp,\psi) = \langle (\pi \circ \overline{R})(\vp), \psi \rangle = \langle \overline{R}(\vp), \psi \rangle -\frac{\tr(\overline{R}(\vp))}{n} \langle g, \psi \rangle  = \overline{R}(\vp,\psi).$$ \\
(4)-(6). See \cite[Proposition 4.1]{Li21}.\\
(7). See \cite[Theorem 1.5]{CGT23} and \cite[Theorem 1.5]{Li22JGA}.\\
(8). See \cite[Theorem 1.6]{Li22JGA}.
\end{proof}

Next, we collect several examples on which the eigenvalues of the curvature operator of the second kind are known explicitly (see \cite{Li22product} for more such examples). These examples are used to demonstrate the sharpness of our results.

\begin{example}\label{example cylinder}
The eigenvalues of $\mathring{R}$ on $\mathbb{S}^{n-1}\times \mathbb{S}^1$ (with the standard product metric) are given by $-\frac{n-2}{n}$ with multiplicity $1$, $0$ with multiplicity $n-1$, and $1$ multiplicity $\frac{(n-2)(n+1)}{2}$; see \cite[Example 2.6]{Li21}. The curvature operator of the second kind of $\mathbb{S}^{n-1}\times \mathbb{S}^1$ lies on $\partial \mathcal{C}(\a,\bar{\Theta}_{n,\a})$, where $\bar{\Theta}_{n,\a}$ is defined in \eqref{eq theta cylinder}. 
\end{example}

\begin{example}\label{example complex projective space}
Bourguignon and Karcher \cite{BK78} computed that the eigenvalues of $\mathring{R}$ on $(\mathbb{CP}^m, g_{FS})$, the complex projective space with the Fubini-Study metric normalized with constant holomorphic sectional curvature $4$, are given by $-2$ with multiplicity $m^2-1$ and $4$ with multiplicity $m(m+1)$. The curvature operator of the second kind of $(\mathbb{CP}^m, g_{FS})$ lies on $\partial \mathcal{C}(\a,B_{m,\a})$, where $B_{m,\a}$ is defined in \eqref{eq def Xi_{m,a}}.
\end{example}

\begin{example}\label{example SXS}
Let $\mathbb{S}^n(\kappa)$ denote the $n$-sphere with constant sectional curvature $\k>0$. According to \cite[Example 2.7]{Li22product}, the eigenvalues of $\mathring{R}$ on $\mathbb{S}^{k}(\k_1)\times \mathbb{S}^{n-k}(\k_2)$ with the standard product metric are given by $-\frac{k(n-k-1)\k_2+(n-k)(k-1)\k_1}{n}$ with multiplicity $1$, $0$ with multiplicity $k(n-k)$, $\k_1$ with multiplicity $\frac{(k-1)(k+2)}{2}$, and $\k_2$ with multiplicity $\frac{(n-k-1)(n-k+2)}{2}$.  
If $\k_1$ and $\k_2$ satisfy $(k-1)\k_1 =(n-k-1)\k_2$ and $2\leq k \leq \frac{n}{2}$, then $\mathbb{S}^{k}(\k_1)\times \mathbb{S}^{n-k}(\k_2)$ is an Einstein manifold and its curvature operator of the second kind lies on $\partial \mathcal{C}(\frac{n+2}{2}, \frac{2(n-1)}{n+2})$. 
\end{example}

\begin{example}\label{example CPXCP}
It was computed in \cite[Example 2.12]{Li22product} that the eigenvalues of $\mathring{R}$ on $\mathbb{CP}^{k}\times \mathbb{CP}^{m-k}$ are given by $-2-\frac{4k(m-k)}{m}$ with multiplicity $1$, $-2$ with multiplicity $k^2+(m-k)^2-2$, $0$ with multiplicity $4k(m-k)$, and $4$ with multiplicity $k(k+1)+(m-k)(m-k+1)$. Here each factor is normalized such that its sectional curvatures lie in the interval $[1,4]$. 
One verifies that the curvature operator of the second kind of $\mathbb{CP}^{k}\times \mathbb{CP}^{m-k}$ lies on $\partial \mathcal{C}(m^2-1, \frac{2m-1}{m+1})$. 

\end{example}

The following identities are useful for calculations. 
\begin{proposition}\label{prop C2K calculation}
Let $(V,g)$ be a Euclidean vector space of dimension $n\geq 3$ and $\{e_i\}_{i=1}^n$ be an orthonormal basis of $V$. Then 
\begin{equation}\label{eq C2K calculation}
    \overline{R}(e_i \odot e_j, e_k \odot e_l) = 2(R_{iklj}+R_{ilkj})
\end{equation}
for all $1\leq i, j ,k ,l \leq n$. In particular, 
\begin{equation*}\label{eq C2K calculation ring}
    \mathring{R}(e_i \odot e_j, e_k \odot e_l) = 2(R_{iklj}+R_{ilkj})
\end{equation*}
for all $1\leq i, j ,k ,l \leq n$ satisfying $i\neq j$ and $k \neq l$. 
\end{proposition}
\begin{proof}
This is a straightforward calculation. See \cite[Lemma 3.1]{Li22Kahler} \footnote{$\mathring{R}$ should be replaced by $\overline{R}$ in \cite[Lemma 3.1]{Li22Kahler}.}. 
\end{proof}

\subsection{An elementary lemma}
The following elementary lemma will be used frequently. 
\begin{lemma}\label{lemma average}
Let $L$ be a positive integer and $A$ be a collection of $L$ real numbers.  Denote by $a_i$ the $i$-th smallest number in $A$ for $1\leq i \leq L$.  
Define a function $f(A,x)$ by 
   \begin{equation*}
       f(A,x)=a_1 + \cdots +a_{\lfloor{x}\rfloor} +(x-\lfloor{x}\rfloor) a_{\lfloor{x}\rfloor+1}, 
   \end{equation*}
for $x\in [1,L]$. Then we have
   \begin{equation}\label{eq function}
       f(A,x) \leq x \bar{a}, 
   \end{equation}
   where $\bar{a}:=\frac{1}{L}\sum_{i=1}^L a_i$ is the average of all numbers in $A$. 
   Moreover, the equality holds for some $x\in [1,L)$ if and only if $a_i=\bar{a}$ for all $1\leq i\leq L$. 
\end{lemma}

\subsection{The cones}
Recall that for $\a \in [1,N)$ and $\theta >-1$, we defined in Definition \ref{def cone} that 
\begin{equation*}
\mathcal{C}(\a,\theta) =\left\{ \mathring{R} \in S^2(S^2_0(V)) : \mathring{R} \text{ satisfies } \eqref{eq C2K} \right\}.
\end{equation*}
The interior and boundary of $\mathcal{C}(\a,\theta)$ are denoted by $\mathring{\mathcal{C}}(\a,\theta)$ and $\partial \mathcal{C}(\a,\theta)$, respectively.

We prove some basic properties of $\mathcal{C}(\a,\theta)$. 
\begin{proposition}\label{prop basic cones}
Let $R \in S^2_B(\wedge^2 V)$ and $\mathring{R}$ be its induced curvature operator of the second kind. Denote by $S$ the scalar curvature of $R$. 
\begin{enumerate}
    \item If $\mathring{R} \in \mathcal{C}(\a,\theta)$, then $S\geq 0$. Moreover, $S=0$ implies $\mathring{R}=0$ and $R=0$. 
    \item If $-\mathring{R} \in \mathcal{C}(\a,\theta)$, then $S\leq 0$. Moreover, $S=0$ implies $\mathring{R}=0$ and $R=0$. 
    \item If $\mathring{R} \in \mathring{\mathcal{C}}(\a,\theta)$, then $S > 0$. 
    \item If $-\mathring{R} \in \mathring{\mathcal{C}}(\a,\theta)$, then $S < 0$. 
\end{enumerate}
\end{proposition}
\begin{proof}
(1) $\tr(\mathring{R})=\frac{n+2}{2n}S$ implies that $\bar{\l}=\frac{S}{n(n-1)}$.
By Lemma \ref{lemma average}, we have 
\begin{equation}\label{eq 2.5}
    \a^{-1}(\l_1 + \cdots +\l_\a) \leq \bar{\l}
\end{equation}
for $\a \in [1,N)$. Moreover, the equality holds if and only if $\l_1= \cdots =\l_N = \bar{\l}$. 

If $\mathring{R} \in \mathcal{C}(\a,\theta)$, then we have $-\theta \bar{\l} \leq \bar{\l}$. Since $\theta >-1$, we must have $\bar{\l} \geq 0$ and therefore $S\geq 0$. Moreover, $\bar{\l}=0$ implies equality in \eqref{eq 2.5}, which forces $\mathring{R}=0$ and then $R=0$. 

(2)-(4). These immediately follow from (1). 
\end{proof}

\begin{proposition}\label{prop cone monotonicity} 
The cones $\mathcal{C}(\a,\theta)$ satisfy the following inclusions:
\begin{enumerate}
    \item If $\a_1 \leq \a_2$, then $\mathcal{C} (\a_1, \theta ) \subset \mathcal{C} (\a_2,\theta)$;
    \item If $\theta_1 \leq \theta_2$, $\mathcal{C} (\a, \theta_1 ) \subset \mathcal{C} (\a,\theta_2)$.
\end{enumerate}

\end{proposition}
\begin{proof}
(1). This follows from Lemma \ref{lemma average}.\\
(2). This is because $\bar{\l}\geq 0$. 
\end{proof}

\section{Ricci Curvature}

In this section, we establish implications of $\mathring{R} \in \mathcal{C}(\a,\theta)$ on the Ricci curvature and then use them to prove Theorem \ref{thm 3D} and Proposition \ref{prop Ricci nonnegative}.

\begin{proposition}\label{prop Ricci}
Let $R \in S^2_B(\wedge^2 V)$ and denote by $\mathring{R}$ its induced curvature operator of the second kind. 
\begin{enumerate}
    \item If $\mathring{R} \in \mathcal{C}(\a, \theta)$ with $1\leq \a \leq n$, then 
    \begin{equation*}
        \Ric \geq \frac{n-1}{\a+1}(1-\a \theta ) \bar{\l} g.
    \end{equation*}
    \item If $\mathring{R} \in \mathcal{C}(\a, \theta)$ with $n \leq \a < N$, then 
    \begin{equation*}
        \Ric \geq (n-1)\frac{n^2-n(\a \theta +\a -1) +2(\a \theta -1)}{n^2+n-2(\a+1)}  \bar{\l} g.
    \end{equation*}
\end{enumerate}
Moreover, strict inequalities hold if we assume $\mathring{R}\in \mathring{\mathcal{C}}(\a, \theta)$.
\end{proposition}

Proposition \ref{prop Ricci nonnegative} follows from Proposition \ref{prop Ricci} by taking $\theta =\bar{\Theta}_{n,\a}$. In addition, we note that Proposition \ref{prop Ricci} recovers several previous results.  

\begin{corollary}[\cite{Li21}, part (2) of Proposition 4.1]
If $\mathring{R}$ is $n$-nonnegative, then $\Ric \geq \frac{S}{n(n+1)} \geq 0$. 
\end{corollary}
\begin{proof}
Take $\a=n$ and $\theta=0$ in Proposition \ref{prop Ricci}.
\end{proof}

\begin{corollary}[\cite{Li22JGA}, Theorem 1.6]
If $\mathring{R}$ is $\left(n+\frac{n-2}{n}\right)$-nonnegative, then the Ricci curvature is nonnegative.
\end{corollary}
\begin{proof}
Take $\a=n+\frac{n-2}{n}$ and $\theta=0$ in Proposition \ref{prop Ricci}.
\end{proof}

We give the proof of Proposition \ref{prop Ricci}. 
\begin{proof}[Proof of Proposition \ref{prop Ricci}]
(1). Let $\{e_i\}_{i=1}^n$ be an orthonormal basis of $V$. Then
\begin{equation*}
    \vp_1 =\tfrac{1}{2\sqrt{n(n-1)}} \left( (n-1)e_1 \odot e_1 -\sum_{p=2}^n e_p \odot e_p \right)
\end{equation*}
and 
\begin{equation*}
    \vp_i =\tfrac{1}{\sqrt{2}} e_1 \odot e_i, \text{ for } 2\leq i \leq n
\end{equation*}
form an orthonormal subset of $S^2_0(V)$ of dimension $n$. 
We may reorder $\vp_i$ for $2\leq i \leq n$ so that
\begin{equation*}
\mathring{R}(\vp_2,\vp_2) \leq \cdots \leq \mathring{R}(\vp_n,\vp_n). 
\end{equation*}
By Lemma \ref{lemma average}, we have
\begin{equation}\label{eq 3.5}
    \sum_{i=2}^{\lfloor{\a}\rfloor}\mathring{R}(\vp_{i},\vp_{i}) + (\a-\lfloor{\a}\rfloor)  \mathring{R}(\vp_{\lfloor{\a}\rfloor+1},\vp_{\lfloor{\a}\rfloor+1}) 
\leq  \frac{\a-1}{n-1} \sum_{i=2}^n \mathring{R}(\vp_{i},\vp_{i}),
\end{equation}
If $\a \in [1,n]$, then $\mathring{R} \in \mathcal{C}(\a, \theta)$ implies
\begin{eqnarray} \label{eq 3.0}
   -\a \theta \bar{\l} &\leq& \l_1+ \cdots +\l_{\a} \\ \nonumber
   &\leq & \sum_{i=1}^{\lfloor{\a}\rfloor}\mathring{R}(\vp_{i},\vp_{i}) + (\a-\lfloor{\a}\rfloor)  \mathring{R}(\vp_{\lfloor{\a}\rfloor+1},\vp_{\lfloor{\a}\rfloor+1}) \\ \nonumber
   &\leq & \mathring{R}(\vp_1,\vp_1) + \frac{\a-1}{n-1} \sum_{i=2}^n \mathring{R}(\vp_{i},\vp_{i}),
\end{eqnarray}
where we have used \eqref{eq 3.5}. 

Using \eqref{eq C2K calculation}, we calculate
\begin{eqnarray*}
   && 4n(n-1) \mathring{R}(\vp_1,\vp_1) \\
   &=& 4n(n-1) \overline{R}(\vp_1,\vp_1) \\
   &=& -2(n-1) \sum_{p=2}^n \overline{R}(e_1\odot e_1, e_p \odot e_p) + \sum_{p, q=2}^n \overline{R}(e_p\odot e_p, e_q \odot e_q) \\
   &=& 8(n-1) \sum_{p=2}^n R_{1p1p} -4 \sum_{p,q=2}^n R_{pqpq} \\
   &=& 8(n-1) R_{11} - 4(S-2R_{11}) \\
   &=& 8nR_{11}-4S,
\end{eqnarray*}
where $S$ denotes the scalar curvature. It follows, by noticing $\bar{\l}=\frac{S}{n(n-1)}$, that
\begin{equation}\label{eq 3.01}
    \mathring{R}(\vp_1,\vp_1) =\frac{2}{n-1}R_{11} - \bar{\l}. 
\end{equation}
Next, we compute
\begin{equation}\label{eq 3.02}
   \sum_{i=2}^n \mathring{R}(\vp_i,\vp_i) = \sum_{i=2}^n  R_{1i1i} =R_{11}.
\end{equation}
Substituting \eqref{eq 3.01} and \eqref{eq 3.02} into \eqref{eq 3.0}, we obtain
\begin{equation*}
 R_{11} \geq \frac{n-1}{\a+1}(1-\a \theta) \bar{\l}.
\end{equation*}
Since the orthonormal frame $\{e_i\}_{i=1}^n$ is arbitrary, we get the desired Ricci lower bound. 

(2). Extend $\{\vp_i\}_{i=1}^n$ in part (1) and to $\{\vp_i\}_{i=1}^N$, an orthonormal basis of $S^2_0(V)$. By reordering $\vp_i$ for $n+1\leq i \leq N$, we may assume that
\begin{equation*}
\mathring{R}(\vp_{n+1},\vp_{n+1}) \leq \cdots \leq \mathring{R}(\vp_N,\vp_N).
\end{equation*}
It follows from Lemma \ref{lemma average} that we have for $\a \in [n,N)$,
\begin{equation}\label{eq 3.20}
\sum_{i=n+1}^{\lfloor{\a}\rfloor}\mathring{R}(\vp_{i},\vp_{i}) + (\a-\lfloor{\a}\rfloor)  \mathring{R}(\vp_{\lfloor{\a}\rfloor+1},\vp_{\lfloor{\a}\rfloor+1}) 
\leq \frac{\a-n}{N-n} \sum_{i=n+1}^N \mathring{R}(\vp_{i},\vp_{i}).
\end{equation}
Here and in the rest of this paper, we use the convention that $\sum_{i=a}^b =0$ whenever $a > b$. 

Using $\mathring{R} \in \mathcal{C}(\a,\theta)$ with $\a \in [n,N)$ and \eqref{eq 3.20}, we obtain 
\begin{eqnarray*}
   -\a \theta \bar{\l} &\leq &   \l_1+ \cdots +\l_{\a} \\
   &\leq & \sum_{i=1}^{n}\mathring{R}(\vp_{i},\vp_{i}) +\sum_{i=n+1}^{\lfloor{\a}\rfloor}\mathring{R}(\vp_{i},\vp_{i}) + (\a-\lfloor{\a}\rfloor)  \mathring{R}(\vp_{\lfloor{\a}\rfloor+1},\vp_{\lfloor{\a}\rfloor+1}) \\
   &\leq & \sum_{i=1}^{n}\mathring{R}(\vp_{i},\vp_{i}) + \frac{\a-n}{N-n} \sum_{i=n+1}^N \mathring{R}(\vp_{i},\vp_{i}).
\end{eqnarray*}
Using \eqref{eq 3.01}, \eqref{eq 3.02}, and 
\begin{equation*}
\sum_{i=n+1}^N \mathring{R}(\vp_{i},\vp_{i}) =N\bar{\l} -   \sum_{i=1}^{n}\mathring{R}(\vp_{i},\vp_{i}) ,
\end{equation*}
we deduce that 
\begin{equation*}
\frac{(n+1)(N-\a)}{(n-1)(N-n)}R_{11} \geq   (1-\a \theta) \bar{\l} -\frac{\a-n}{N-n}(N+1)\bar{\l}.
\end{equation*}
It follows that
\begin{equation*}
    R_{11} \geq (n-1)\frac{n^2-n(\a \theta +\a -1) +2(\a \theta -1)}{n^2+n-2(\a+1)}  \bar{\l}. 
\end{equation*}
The Ricci lower bound follows immediately as the orthonormal frame $\{e_i\}_{i=1}^n$ is arbitrary. 

This finishes the proof. 
\end{proof}

\begin{remark}
The idea of the above proof is to use $\mathbb{S}^{n-1}\times \mathbb{S}^1$ as a model space. If $e_1$ is in the tangent space of the $\mathbb{S}^1$ factor, then the chosen $\{\vp_i\}_{i=1}^N$ are the eigentensors of $\mathring{R}$ on $\mathbb{S}^{n-1}\times \mathbb{S}^1$. 
\end{remark}

Next, we prove Theorem \ref{thm 3D}.

\begin{proof}[Proof of Theorem \ref{thm 3D}]
(1). By Proposition \ref{prop Ricci nonnegative}, if $(M,g)$ satisfies $\mathring{R} \in \mathring{\mathcal{C}}(\a, \bar{\Theta}_{3,\a})$, then $(M,g)$ has positive Ricci curvature. By Hamilton's famous work \cite{Hamilton82}, we conclude that $M$ is diffeomorphic to a spherical space form. 

(2). By Proposition \ref{prop Ricci nonnegative}, the assumption implies that $M$ has nonnegative Ricci curvature. The classification then follows from Hamilton's classification of closed three-manifolds with nonnegative Ricci curvature in \cite{Hamilton82, Hamilton86}. 
\end{proof}

Finally, we remark that, by Proposition \ref{prop Ricci nonnegative} and Liu's classification result \cite{Liu13}, a complete noncompact three-manifold satisfying $\mathring{R} \in \mathcal{C}(\a,\bar{\Theta}_{3,\a})$ is either diffeomorphic to $\R^3$ or its universal cover is isometric to $N^2 \times \R$, where $N^2$ is a complete surface with nonnegative scalar curvature.

\section{Isotropic Curvature}

In this section, we explore the implication of $\mathring{R} \in \mathcal{C}(\a,\theta)$ on the isotropic curvatures in dimension four. We first recall the definition of isotropic curvature.
\begin{definition}
Let $(V,g)$ be a Euclidean vector space of dimension $n\geq 4$. $R \in S^2_B(\wedge^2 V)$ is said to have nonnegative isotropic curvature if 
\begin{equation*}
    R_{1313}+R_{1414}+R_{2323}+R_{2424}-2R_{1234} \geq 0
\end{equation*}
for any orthonormal four-frame $\{e_1,e_2,e_3,e_4\} \subset V$.
If the strict inequality holds, then $R$ is said to have positive isotropic curvature. 
\end{definition}

The main result of this section states
\begin{proposition}\label{prop PIC}
Let $(V,g)$ be a Euclidean vector space of dimension $4$. Let $R \in S^2_B(\wedge^2 V)$ and denote by $\mathring{R}$ its induced curvature operator of the second kind. 
\begin{enumerate}
    \item If $\mathring{R} \in \mathcal{C}(\a,1)$ with $1\leq \a \leq 3$ or $\mathring{R} \in \mathcal{C}(\a,9\a^{-1}-2)$ with $3\leq \a < 9$, then $R$ has nonnegative isotropic curvature.
    \item If $\mathring{R} \in \mathring{\mathcal{C}}(\a,1)$ with $1\leq \a \leq 3$ or $\mathring{R} \in \mathring{\mathcal{C}}(\a,9\a^{-1}-2)$ with $3\leq \a < 9$, then $R$ has positive isotropic curvature. 
\end{enumerate}
\end{proposition}

\begin{proof}
(1). Let $\{e_1, e_2, e_3, e_4\}$ be an orthonormal basis of $V$. Define traceless symmetric two-tensors 
\begin{eqnarray*}
\vp_1 &=& \tfrac{1}{4}\left(e_1\odot e_1 +e_2 \odot e_2-e_3 \odot e_3 -e_4 \odot e_4 \right), \\
\vp_2 &=& \tfrac{1}{2}\left(e_1 \odot e_4 -e_2 \odot e_3 \right), \\
\vp_3 &=& \tfrac{1}{2}\left(e_1 \odot e_3 +e_2 \odot e_4 \right).
\end{eqnarray*}
Then $\{\vp_i\}_{i=1}^3$ form an orthonormal subset of $S^2_0(V)$. A straightforward computation using \eqref{eq C2K calculation} produces
\begin{eqnarray*}
2 \mathring{R}(\vp_1,\vp_1) &=& -R_{1212}-R_{3434} +R_{1313}+R_{2424}+R_{1414}+R_{2323}, \\
2 \mathring{R}(\vp_2,\vp_2) &=& R_{1414}+R_{2323}-2R_{1234}+2R_{1342}, \\
2 \mathring{R}(\vp_3,\vp_3) &=& R_{1313}+R_{2424}-2R_{1234}+2R_{1423}. 
\end{eqnarray*}
Together with the first Bianchi identity, we get
\begin{eqnarray}\label{eq 4.1}
    \sum_{i=1}^3 \mathring{R}(\vp_i,\vp_i) &=& R_{1313}+R_{1414}+R_{2323}+R_{2424} \\ \nonumber
    &&-\tfrac{1}{2}(R_{1212}+R_{3434}) -3R_{1234}.
\end{eqnarray}
If $\mathring{R} \in \mathcal{C}(\a, 1)$ with $1\leq \a \leq 3$, then Lemma \ref{lemma average} implies
\begin{equation}\label{eq 4.2}
-\bar{\l} \leq  \a^{-1} \left( \l_1 + \cdots +\l_{\a} \right)
\leq  \tfrac{1}{3}(\l_1+\l_2+\l_3) 
\leq  \tfrac{1}{3} \sum_{i=1}^3 \mathring{R}(\vp_i,\vp_i).
\end{equation}
Note that in dimension four, we have
\begin{equation}\label{eq 4.3}
    \bar{\l}= \tfrac{S}{12}=\tfrac{1}{6}\left(R_{1313}+R_{1414}+R_{2323}+R_{2424}+R_{1212}+R_{3434} \right).
\end{equation}
Substituting \eqref{eq 4.1} and \eqref{eq 4.3} into \eqref{eq 4.2} produces
\begin{equation}\label{eq 4.21}
R_{1313}+R_{1414}+R_{2323}+R_{2424} -2 R_{1234} \geq 0.
\end{equation}
Since the orthonormal four-frame $\{e_1, e_2, e_3, e_4\}$ is arbitrary, we conclude that $R$ has nonnegative isotropic curvature. 

To handle the case $3\leq \a < 9$, we extend $\{\vp_i\}_{i=1}^3$ to an orthonormal basis $\{\vp_i\}_{i=1}^9$ of $S^2_0(V)$, and reorder $\vp_i$ for $4\leq i\leq 9$ such that
\begin{equation*}
\mathring{R}(\vp_4,\vp_4) \leq \cdots \leq \mathring{R}(\vp_9,\vp_9).
\end{equation*}
By Lemma \ref{lemma average}, this ordering implies for $3 \leq \a <9$,
\begin{equation}\label{eq 4.20}
\sum_{i=4}^{\lfloor{\a}\rfloor} \mathring{R}(\vp_i,\vp_i) + (\a -\lfloor{\a}\rfloor) \mathring{R}(\vp_{\lfloor{\a}\rfloor+1}, \vp_{\lfloor{\a}\rfloor+1}) \leq \frac{\a-3}{6} \sum_{i=4}^9 \mathring{R}(\vp_i,\vp_i).
\end{equation}
If $\mathring{R} \in \mathcal{C}(\a, 9\a^{-1}-2)$ with $3\leq \a < 9$, then
\begin{eqnarray*}
 -(9-2\a) \bar{\l} 
&\leq & \l_1 + \cdots +\l_{\a}  \\ 
&\leq & \sum_{i=1}^{\lfloor{\a}\rfloor} \mathring{R}(\vp_i,\vp_i) + (\a -\lfloor{\a}\rfloor) \mathring{R}(\vp_{\lfloor{\a}\rfloor+1}, \vp_{\lfloor{\a}\rfloor+1}) \\
&\leq & \sum_{i=1}^3 \mathring{R}(\vp_i,\vp_i) + \frac{\a-3}{6} \sum_{i=4}^9 \mathring{R}(\vp_i,\vp_i) \\
&=& \frac{9-\a}{6} \sum_{i=1}^3 \mathring{R}(\vp_i,\vp_i) + \frac{3(\a-3)}{2} \bar{\l},
\end{eqnarray*}
where we have used \eqref{eq 4.20} and 
$$\sum_{i=1}^9 \mathring{R}(\vp_i,\vp_i) =9 \bar{\l}.$$ 
The inequality simplifies as
\begin{equation*}
    -3\bar{\l} \leq  \sum_{i=1}^3 \mathring{R}(\vp_i,\vp_i),
\end{equation*}
which, after substituting into \eqref{eq 4.1} and \eqref{eq 4.3}, yields \eqref{eq 4.21}. 
Since the orthonormal four-frame $\{e_1, e_2, e_3, e_4\}$ is arbitrary, this proves that $R$ has nonnegative isotropic curvature. 

(2). This is similar to (1). If $\mathring{R} \in \mathring{\mathcal{C}}(\a,\theta)$, then some of the inequalities become strict and we obtain strict inequality in \eqref{eq 4.21}, proving that $R$ has positive isotropic curvature 
\end{proof}

\begin{remark}
The above proof uses $\mathbb{CP}^2$ as a model space. For a suitably chosen orthonormal frame $\{e_1,e_2,e_3,e_4\}$, $\spn\{\vp_1,\vp_2,\vp_3\}$ is the eigenspace associated with the eigenvalue $-2$ of $\mathring{R}$ on $\mathbb{CP}^2$ and its orthogonal complement is the eigenspace associated with the eigenvalue $4$. See \cite{BK78} or \cite{Li22Kahler}. 
\end{remark}

Next, we prove Theorem \ref{thm 4D}.
\begin{proof}[Proof of Theorem \ref{thm 4D}]
(1). By Proposition \ref{prop Ricci nonnegative}, $(M,g)$ has positive Ricci curvature. Note that $\Theta_{4,\a} \leq 1$ for $1\leq \a \leq 3$ and $\Theta_{4,\a} \leq 9\a^{-1}-2$ for $3\leq \a < 9$. By Proposition \ref{prop cone monotonicity}, we have $\mathcal{C}(\a,\Theta_{4,\a}) \subset \mathcal{C}(\a,1)$ if $1\leq \a \leq 3$ and $\mathcal{C}(\a,\Theta_{4,\a}) \subset \mathcal{C}(\a,9\a^{-1}-2)$ if $3\leq \a <9$. Proposition \ref{prop PIC} then implies that $(M,g)$ has positive isotropic curvature. The work of Hamilton \cite{Hamilton97} implies that $M$ is diffeomorphic to a spherical space form. 

(2). Similar to in part (1), we use Proposition \ref{prop Ricci nonnegative} and Proposition \ref{prop PIC} to get that $(M,g)$ has nonnegative Ricci curvature and nonnegative isotropic curvature. 
Denote by $(\widetilde{M},\tilde{g})$ the universal cover of $(M,g)$. By the Cheeger-Gromoll theorem (see \cite[Theorem 7.3.11]{Petersen2016book}), $(\widetilde{M},\tilde{g})$ splits isometrically as a product $(N^{4-k},g_N) \times \R^{k}$, where $N^{4-k}$ is a closed manifold. Note that $(\widetilde{M},\tilde{g})$ is flat if $k=3$ or $4$.

We show that $(\widetilde{M},\tilde{g})$ is flat if $k=2$. More generally, we prove that if $(\widetilde{M},\tilde{g})$ splits isometrically as the product of $(N_1,g_1)\times (N_2,g_2)$ with $\dim(N_1)=\dim(N_2)=2$, then $(\widetilde{M},\tilde{g})$ must be flat. According to \cite[Proposition 2.1]{Li22product}, the eigenvalues of $\mathring{R}$ on $(\widetilde{M}, \tilde{g})=(N_1,g_1)\times (N_2,g_2)$ are given by 
\begin{equation*}
 \left\{-\tfrac{S_1+S_2}{4}, 0, 0, 0, 0,   \tfrac{S_1}{2}, \tfrac{S_1}{2}, \tfrac{S_2}{2}, \tfrac{S_2}{2} \right\},
\end{equation*}
where $S_i$ denotes the scalar curvature of $N_i$ for $i=1,2$. 
Since $(\widetilde{M},\tilde{g})$ has nonnegative Ricci curvature, we infer that $S_1\geq 0$ and $S_2 \geq 0$. The condition $\mathring{R}\in \mathcal{C}(\a,\bar{\Theta}_{4,\a})$ implies 
\begin{equation*}
    -\a \bar{\Theta}_{4,\a} \bar{\l} \leq \begin{cases}
        -\frac{S_1+S_2}{4}, & 1\leq \a \leq 5;\\
        -\frac{S_1+S_2}{4} +\frac{(\a-5)}{2} \min\{S_1,S_2\}, & 5\leq \a \leq 7; \\
        -\frac{S_1+S_2}{4}+\min\{S_1,S_2\} + \frac{(\a-7)}{2} \max \{S_1,S_2\}, & 7 \leq \a <9. 
    \end{cases}
\end{equation*}
In view of $\bar{\l}=\frac{S_1+S_2}{12}$, one deduces from the above inequality that $S_1=S_2=0$. Therefore, $(\widetilde{M},\tilde{g})$ is flat. 

Next, we examine the case $k=1$. By \cite[Proposition 2.1]{Li22product}, the eigenvalues of $\mathring{R}$ on $(\widetilde{M},\tilde{g}) =(N^3, g_N) \times \R$ are given by
\begin{equation*}
    \left\{-\tfrac{S_N}{12}, 0, 0, 0, \mu_1, \mu_2, \mu_3, \mu_4, \mu_5  \right\}
\end{equation*}
where $S_N$ denotes the scalar curvature of $N$ and $\mu_1 \leq \cdots \leq \mu_5$ denote the eigenvalues of the curvature operator of the second kind of $N$. Note that $(N,g_N)$ is locally irreducible, as it cannot split out another factor of $\R$. Hence, $(N,g_N)$ is locally irreducible, simply connected, and has nonnegative Ricci curvature. By \cite{Hamilton86}, $N$ is diffeomorphic to $\mathbb{S}^3$. Thus, $(\widetilde{M},\tilde{g})$ is either flat or diffeomorphic to $\mathbb{S}^3 \times \R$.

For $k=1$ and $4< \a <9$, we can further conclude that $(N,g_N)$ has constant positive sectional curvature. Noticing $\a \bar{\Theta}_{4,\a}=9-2\a$, $\bar{\l}=\frac{S_N}{12}$, and 
\begin{equation}\label{eq 4.11}
\mu_1 + \cdots + \mu_{\a-4} \leq \frac{\a -4 }{5} \sum_{i=1}^{5} \mu_i =\frac{\a -4 }{5} \left(9\bar{\l} + \frac{S_N}{12}\right), 
\end{equation}
we deduce from $\mathring{R}\in \mathcal{C}(\a,\bar{\Theta}_{4,\a})$ that
\begin{equation*}
  0 \leq  (9-2\a)\tfrac{S_N}{12}  -\tfrac{S_N}{12} + \mu_1 + \cdots + \mu_{\a-4}    \leq 0.  
\end{equation*}
Therefore, \eqref{eq 4.11} attains equality, which happens only when $\mu_1=\mu_2=\cdots =\mu_5$. Hence, $(N,g)$ has pointwise constant sectional curvature. By Schur's lemma, $(N,g)$ is isometric to $\mathbb{S}^3$ with constant positive sectional curvature

At last, we investigate the case $k=0$. Note that $(\widetilde{M}, \tilde{g})=(N^4,g_N)$ is closed, simply connected, irreducible unless it is flat, and has nonnegative isotropic curvature. 
By \cite[Theorem 9.30]{Brendle10book}, $(\widetilde{M}, \tilde{g})$ is either homeomorphic to $\mathbb{S}^4$, or K\"ahler and biholomorphic to $\mathbb{CP}^2$, or isometric to a symmetric space (either $\mathbb{S}^4$ or $\mathbb{CP}^2$). In the first case, the homeomorphism can be upgraded to diffeomorphism using Hamilton's work \cite{Hamilton97} while in the K\"ahler case, it must be flat if $1\leq \a \leq 4$ and isometric to $\mathbb{CP}^2$ if $4\leq \a <9$ by Theorem \ref{thm Kahler flat}.

In summary, we have proved that $\widetilde{M}$ is either flat, or diffeomorphic to $\mathbb{S}^4$, or diffeomorphic to $\mathbb{S}^3\times \R$ (isometric if $4<\a <9$), or isometric to  $\mathbb{CP}^2$ if $4\leq \a <9$. 

\end{proof}

By Proposition \ref{prop PIC} and the classification of closed four-manifolds with positive isotropic curvature (see \cite{Hamilton97}, \cite{CZ06}, and \cite{CTZ12}), we get 
\begin{theorem}\label{thm 4D PIC}
Let $(M^4,g)$ be a closed Riemannian manifold of dimension four satisfying $\mathring{R}\in \mathring{\mathcal{C}}(\a, 1)$ with $1\leq \a \leq 3$ or $\mathring{R} \in \mathring{\mathcal{C}}(\a, 9\a^{-1}-2)$ with $3\leq \a <9$. 
Then $M$ has positive isotropic curvature and $M$ is diffeomorphic to a $\mathbb{S}^4$, $\mathbb{RP}^4$, $(\mathbb{S}^3 \times \R)/G$, or a connected sum of them, where $G$ is a discrete group of isometries of the standard $\mathbb{S}^3 \times  \R$ acting co-compactly without fixed points. 
\end{theorem}

\section{Homological Sphere Theorems}

In this section, we use the Bochner technique to prove Theorem \ref{thm homology sphere}, Theorem \ref{thm Betti}, and Theorem \ref{thm homology sphere Einstein}. 

Recall that a harmonic $p$-form $\omega$ on $(M^n,g)$ satisfies the Bochner formula
\begin{equation}\label{eq Bochner}
    \frac{1}{2}\Delta |\omega|^2 =|\nabla \omega|^2 +g(\Ric_L(\omega), \omega). 
\end{equation}
It was discovered by Nienhaus, Petersen, and Wink \cite{NPW22} that the curvature term $g(\Ric_L(\omega), \omega)$ can be written as 
\begin{equation}\label{eq Bochner curvature term}
\frac{3}{2}g(\Ric_L(\omega), \omega) = \sum_{\a=1}^{N} \l_{\a} |S_\a \omega|^2 +\frac{p(n-2p)}{n}\Ric(\omega, \omega) +\frac{p^2}{n^2}S|\omega|^2, 
\end{equation}
where $\l_1 \leq \cdots \leq \l_N$ are the eigenvalues of $\mathring{R}$, $\{S_{\a}\}_{\a=1}^N$ are the associated eigentensors, and 
\begin{equation*}
\Ric(\omega, \omega) =\sum_{j,k}\sum_{i_2,\ldots, i_n} R_{jk}\omega_{ji_2\ldots i_n}\omega_{ki_2\ldots i_n}.
\end{equation*}
The action of $S_\a$ on $\omega$ is given by (see \cite[Definition 1.3]{NPW22} \footnote{\cite[Definition 1.3]{NPW22} missed the minus sign, as pointed out to us by the authors.}
\begin{equation*}
    (S_\a \omega)(X_1, \cdots, X_p)=-\sum_{k=1}^p \omega(X_1, \cdots, S_\a X_k, \cdots, X_p). 
\end{equation*}
The key of the Bochner technique is to show the nonnegativity of $g(\Ric_L(\omega), \omega)$ under appropriate curvature conditions and then apply the maximum principle to \eqref{eq Bochner}. It is easy to see that $\mathring{R}\geq 0$ implies the nonnegativity of each term on the right-hand side of \eqref{eq Bochner curvature term}. Below we will use the weight principle \cite[Theorem 3.6]{NPW22} and lower estimates on $\Ric(\omega, \omega)$ to show that $g(\Ric_L(\omega), \omega)$ is nonnegative under much weaker curvature conditions. 

We begin with the Einstein case and prove Theorem \ref{thm homology sphere Einstein}. 
\begin{proof}[Proof of Theorem \ref{thm homology sphere Einstein}]

Passing to the orientation double cover if necessary, we may assume that $(M,g)$ is oriented. By Poincar\'e duality, we may assume $1\leq p \leq \frac{n}{2}$. 
Using the Einstein condition $\Ric=\frac{S}{n}g$, the identity $S=n(n-1)\bar{\l}$,  and the identity (\cite[Lemma 3.7, part (a)]{NPW22})
\begin{equation}\label{eq omega norm}
|\omega|^2=\frac{2n}{p(n-p)(n+2)}\sum_{\a=1}^{N} |S_\a \omega|^2,
\end{equation}
we obtain from \eqref{eq Bochner curvature term} that 
\begin{equation}\label{eq 5.1}
    \frac{3}{2}g(\Ric_L(\omega), \omega)= \sum_{\a=1}^{N} \left(\l_{\a}  +\frac{2(n-1)}{n+2}\bar{\l} \right) |S_\a \omega|^2.
\end{equation}
By \cite[Lemma 3.7, part (b)]{NPW22} and the weight principle \cite[Theorem 3.6]{NPW22}, we conclude that if the operator $\mathring{R} + \beta \bar{\l} \id$ is $\frac{n+2}{2}$-nonnegative, then 
$$\sum_{\a=1}^{N} \left(\l_{\a}  +\beta \bar{\l} \right) |S_\a \omega|^2 \geq 0.$$
Below we prove part (3) first and then part (1) and part (2).

(3).  Note that $\mathring{R} \in \mathcal{C}\left(\frac{n+2}{2}, \frac{2(n-1)}{n+2} \right)$ if and only if the operator $\mathring{R} + \frac{2(n-1)}{n+2}\bar{\l} \id$ is $\frac{n+2}{2}$-nonnegative. 
Therefore, 
$$\frac{3}{2}g(\Ric_L(\omega), \omega)=\sum_{\a=1}^{N} \left(\l_{\a}  +\frac{2(n-1)}{n+2}\bar{\l} \right) |S_\a \omega|^2 \geq 0.$$
Applying the maximum principle to \eqref{eq Bochner} yields that $\omega$ must be parallel. 

(1). By (3), $\omega$ must be parallel. It follows from \eqref{eq Bochner} and \eqref{eq 5.1} that
\begin{equation*}
    \sum_{\a=1}^{N} \left(\l_{\a}  +\frac{2(n-1)}{n+2}\bar{\l} \right) |S_\a \omega|^2=0.
\end{equation*}
If $\mathring{R} \in \mathring{\mathcal{C}}\left(\frac{n+2}{2}, \frac{2(n-1)}{n+2} \right)$ and $\omega$ does not vanish, then the left-hand side becomes strictly positive, yielding a contradiction. Thus, $\omega \equiv 0$. Since $b_p(M,\R)=0$ for all $1\leq p\leq \frac{n}{2}$, we conclude that $M$ is a rational homology sphere. 

(2). By (3), $\omega$ must be parallel and 
\begin{eqnarray*}
    0  &=& \sum_{\a=1}^{N} \left(\l_{\a}  +\frac{2(n-1)}{n+2}\bar{\l} \right) |S_\a \omega|^2 \\
    &=& \sum_{\a=1}^{N} \left(\l_{\a}  +\theta \bar{\l} \right) |S_\a \omega|^2 + \left(\frac{2(n-1)}{n+2}-\theta \right) \bar{\l} \sum_{\a=1}^{N} |S_\a \omega|^2. 
\end{eqnarray*}
The first term is nonnegative by $\mathring{R}\in \mathcal{C}\left(\frac{n+2}{2},\theta \right)$ and the second term is nonnegative by $\bar{\l}\geq 0$ and $\theta < \frac{2(n-1)}{n+2}$. It follows that both terms are equal to zero on $M$ and at every point in $M$ we have either $\bar{\l}=0$ or $\omega=0$. If $\bar{\l}=0$ at a point, then $M$ is scalar flat everywhere and hence flat by Proposition \ref{prop basic cones}. Otherwise, we have $\bar{\l} >0$ everywhere and $\omega \equiv 0$. Hence, $(M,g)$ is either flat or a rational homology sphere. 

This completes the proof of Theorem \ref{thm homology sphere Einstein}.
\end{proof}

Next, we turn to the proof of Theorem \ref{thm Betti}. Without the Einstein condition, we need to estimate the term $\Ric(\omega,\omega)$ from below. As observed in \cite{NPW22}, with respect to an orthonormal basis $\{e_i\}_{i=1}^n$ that diagonalizes the Ricci tensor, we have
\begin{equation*}
\Ric(\omega, \omega)= \sum_{j,k} \sum_{i_2, \cdots, i_p} R_{jk} \omega_{ji_2\cdots i_p} \omega_{ji_2\cdots i_p} = \frac{1}{p}\sum_{I=(i_1, \cdots, i_p)} \left(\sum_{i\in I } R_{ii} \right) \omega_I^2.
\end{equation*}
For this purpose, we establish a lower bound for $\sum_{i=1}^p R_{ii}$. 
\begin{proposition}\label{prop Ricci p}
Let $R \in S^2_B(\wedge^2 V)$ and denote by $\mathring{R}$ its induced curvature operator of the second kind. If $1\leq p \leq \frac{n}{2}$ and $\mathring{R} \in \mathcal{C}\left(\frac{(n-1)p}{2}, \theta \right)$, then 
\begin{equation*}
    \sum_{i=1}^p R_{ii} \geq \frac{(n-1)p}{n-p+2} (1-(n-p+1)\theta) \bar{\l}
\end{equation*}
for any orthonormal basis $\{e_i\}_{i=1}^n$ of $V$. 
\end{proposition}

\begin{proof}
Let $\{e_i\}_{i=1}^n$ be an orthonormal basis of $V$. Define traceless symmetric two-tensors 
\begin{equation*}
\vp_{ij} =\tfrac{1}{\sqrt{2}} e_i \odot e_j, \text{ for } 1\leq i < j \leq n,
\end{equation*}
and 
\begin{equation*}
\psi_k = \tfrac{1}{2\sqrt{(n-k)(n-k+1)}} \left((n-k)e_k \odot e_k -\sum_{l=k+1}^n e_l \odot e_l \right),
\end{equation*}
for $1 \leq k \leq n-1$. Then $\{\vp_{ij}\}_{1\leq i < j \leq n} \cup \{\psi_k\}_{1\leq k \leq n-1}$ form an orthonormal basis of $S^2_0(V)$. 
For simplicity of notation, we set 
\begin{eqnarray*}
    a_{ij} &:=& \mathring{R}(\vp_{ij}, \vp_{ij}), \\
    b_k &:=& \mathring{R}(\psi_k,\psi_k),
\end{eqnarray*}
and 
\begin{eqnarray*}
Q:= 2(n-p+1) \sum_{1\leq i < j \leq p} a_{ij} +(n-p) \sum_{i=1}^p \sum_{j=p+1}^n a_{ij}  + (n-p) \sum_{k=1}^p b_k.  
\end{eqnarray*}
It was observed in \cite[page 23]{NPW22} that 
\begin{equation}\label{eq Q and p-Ricci}
    Q= (n-p+2)\sum_{i=1}^p R_{ii} -(n-1)p\bar{\l}. 
\end{equation}
Therefore, we need to bound $Q$ from below. 

Noticing
\begin{eqnarray*}
    (n-p-1) \sum_{i=1}^p \sum_{j=p+1}^n a_{ij} = \sum_{l=p+1}^n \sum_{i=1}^p \sum_{\substack{p+1 \leq j \leq n \\ j\neq l}} a_{ij},
\end{eqnarray*}
we arrange, when $p$ is even, that
\begin{eqnarray*}
Q &=& \sum_{l=p+1}^n \left(\sum_{1\leq i < j \leq p} a_{ij} + \sum_{k=1}^{\frac{p}{2}} b_k + \sum_{i=1}^{\frac{p}{2}} \sum_{\substack{p+1 \leq j \leq n \\ j \neq l}} a_{ij} \right) \\
&& + \sum_{l=p+1}^n \left(\sum_{1\leq i < j \leq p} a_{ij} + \sum_{k=\frac{p}{2}+1}^{p} b_k + \sum_{i=\frac{p}{2}+1}^{p} \sum_{\substack{p+1 \leq j \leq n \\ j \neq l}} a_{ij} \right) \\
&& +\left(\sum_{1\leq i < j \leq p} a_{ij} +\sum_{i=1}^{\frac{p}{2}} \sum_{p+1 \leq j \leq n} a_{ij} \right)  \\
&& + \left( \sum_{1\leq i < j \leq p} a_{ij} +\sum_{i=\frac{p}{2}+1}^{p} \sum_{p+1 \leq j \leq n} a_{ij} \right).
\end{eqnarray*}
Note that in the above arrangement, each bracket contains the sum of $\frac{(n-1)p}{2}$ many terms and the involved symmetric two-tensors are orthonormal. 
Therefore, each bracket is bounded from below by $\l_1+\cdots +\l_{\frac{(n-1)p}{2}}$ and we get
\begin{equation}\label{eq Q lower bound}
    Q \geq 2(n-p+1)\left(\l_1+\cdots +\l_{\frac{(n-1)p}{2}} \right). 
\end{equation}

Next, we show that such an arrangement can also be made when $p$ is odd. Suppose both $p$ and $n$ are odd. Let $E_1$ and $E_2$ be a partition of the set $\{1,2, \cdots, p\}$ with $|E_1|=\frac{p+1}{2}$ and $|E_2|=\frac{p-1}{2}$. For each $p+1\leq l \leq n$, let $F_{1l}$ and $F_{2l}$ be a partition of the set 
$$\left\{(i,j) : 1\leq i \leq p, p+1 \leq j \leq n, j \neq l\right\}$$ 
with $|F_{1l}|=\frac{p(n-p-1)-1}{2}$ and $|F_2|=\frac{p(n-p-1)+1}{2}$. 
Then, we have
\begin{eqnarray*}
Q &=& \sum_{l=p+1}^n \left(\sum_{1\leq i < j \leq p} a_{ij} + \sum_{k \in E_1} b_k + \sum_{(i,j) \in F_{1l}}a_{ij} \right) \\
&& + \sum_{l=p+1}^n \left(\sum_{1\leq i < j \leq p} a_{ij} + \sum_{k \in E_2} b_k + \sum_{(i,j) \in F_{2l}} a_{ij} \right) \\
&& +\left(\sum_{1\leq i < j \leq p} a_{ij} +\sum_{i=1}^{p} \sum_{j=p+1}^{\frac{n+p}{2}+1} a_{ij} \right)  \\
&& + \left( \sum_{1\leq i < j \leq p} a_{ij} +\sum_{i=1}^{p} \sum_{j=\frac{n+p}{2}+2}^{n} a_{ij} \right).
\end{eqnarray*}
It follows that \eqref{eq Q lower bound} holds in this case. 

Suppose $p$ is odd and $n$ is even. Note that $\frac{(n-1)p}{2}=[\frac{(n-1)p}{2}] + \frac{1}{2}$. We can arrange the terms in $Q$ as
\begin{eqnarray*}
 Q &=& \sum_{l=p+1}^n \left(\sum_{1\leq i < j \leq p} a_{ij} + \sum_{k=1}^{\frac{p-1}{2}} b_k + \sum_{(i,j) \in G_{1l}}a_{ij} + \frac{1}{2} b_p \right) \\
&& + \sum_{l=p+1}^n \left(\sum_{1\leq i < j \leq p} a_{ij} + \sum_{k=\frac{p+1}{2}}^{p-1} b_k + \sum_{(i,j) \in G_{2l}}a_{ij} + \frac{1}{2} b_p \right) \\
&& +\left(\sum_{1\leq i < j \leq p} a_{ij}  +\sum_{i=1}^{p} \sum_{j=p+1}^{\frac{n+p-1}{2}} a_{ij}  + \frac{1}{2} a_{pn} \right)  \\
&& + \left( \sum_{1\leq i < j \leq p} a_{ij} + \sum_{i=1}^{p} \sum_{j=\frac{n+p+1}{2}}^{n-1} a_{ij} + \frac{1}{2} a_{pn} \right),   
\end{eqnarray*}
where $G_{1l}$ and $G_{2l}$ is a partition of the set 
$$\{(i,j) | 1\leq i\leq p, p+1 \leq j \leq n, j \neq l\}$$ 
with $|G_{1l}|=|G_{2l}|=\frac{p(n-p-1)}{2}$. This proves \eqref{eq Q lower bound} in this case. 

Using $\mathring{R} \in \mathcal{C}\left(\frac{(n-1)p}{2}, \theta \right)$ and \eqref{eq Q lower bound}, we obtain
\begin{equation*}
Q \geq -(n-1)p(n-p+1)\theta \bar{\l}. 
\end{equation*}
By \eqref{eq Q and p-Ricci}, we then infer that
\begin{equation*}
    \sum_{i=1}^p R_{ii} \geq \frac{(n-1)p}{n-p+2} (1-(n-p+1)\theta) \bar{\l}.
\end{equation*}
This proves Proposition \ref{prop Ricci p}. 
\end{proof}

We are ready to prove Theorem \ref{thm Betti}. 
\begin{proof}[Proof of Theorem \ref{thm Betti}]
Since $\frac{(n-1)p}{2} >  \frac{n+2}{2}$ for $p\geq 2$, we have $$\mathcal{C}\left(\frac{n+2}{2}, A_{n,p} \right) \subset \mathcal{C} \left(\frac{(n-1)p}{2}, A_{n,p} \right) .$$
By Proposition \eqref{prop Ricci p}, $\mathring{R} \in \mathcal{C}\left(\frac{n+2}{2}, A_{n,p} \right)$ implies
\begin{equation*}
    \sum_{k=1}^p R_{i_ki_k} \geq \frac{(n-1)p}{n-p+2} (1-(n-p+1)A_{n,p}) \bar{\l}.
\end{equation*}
Therefore, we obtain
\begin{eqnarray*}
&&  \frac{p^2}{n^2}S |\omega|^2 +\frac{p(n-2p)}{n}\Ric(\omega, \omega) \\
&=& \frac{p^2(n-1)}{n}\bar{\l} |\omega|^2 +\frac{(n-2p)}{n} \sum_{I=(i_1,\cdots, i_p)} \left(\sum_{i\in I } R_{ii} \right) \omega_I^2 \\
&\geq& \frac{p^2(n-1)}{n}\bar{\l} |\omega|^2 + \frac{(n-2p)(n-1)p}{n(n-p+2)} (1-(n-p+1)A_{n,p}) \bar{\l} |\omega|^2 \\
&= & \frac{2(n-1) \bar{\l} }{(n-p)(n+2)}\left( p+ \frac{(n-2p)}{(n-p+2)} (1-(n-p+1)A_{n,p}) \right) \sum_{\a=1}^N |S_\a w|^2 \\
&=& A_{n,p} \bar{\l} \sum_{\a=1}^N |S_\a w|^2,
\end{eqnarray*}
where we have used \eqref{eq omega norm} and \eqref{eq A_{n,p} def} in getting the last two steps, respectively. 
By \eqref{eq Bochner curvature term}, we have
\begin{equation*}
    \frac{3}{2}g(\Ric_L(\omega), \omega)) \geq \sum_{\a=1}^N  \left(\l_\a+ A_{n,p} \bar{\l}\right) |S_\a w|^2
\end{equation*}

(3). If $\mathring{R} \in \mathcal{C}\left(\frac{n+2}{2}, A_{n,p}\right)$, then the operator $\mathring{R}+A_{n,p}\bar{\l}\id$ is $\frac{n+2}{2}$-nonnegative. The weight principle in \cite[Theorem 3.6]{NPW22} then implies that 
\begin{equation*}
    \frac{3}{2}g(\Ric_L(\omega), \omega)) \geq \sum_{\a=1}^N  \left(\l_\a+ A_{n,p} \bar{\l}\right) |S_\a w|^2 \geq 0.
\end{equation*}
The result follows from the maximum principle. 

(1). This can be proved similarly as in the proof of part (1) of Theorem \ref{thm homology sphere Einstein}. 

(2). By (3), $\omega$ is parallel and 
\begin{eqnarray*}
    0 &=& \frac{3}{2}g(\Ric_L(\omega), \omega)) \\
    &\geq & \sum_{\a=1}^N  \left(\l_\a+ A_{n,p} \bar{\l}\right) |S_\a w|^2  \\
    &=& \sum_{\a=1}^N  \left(\l_\a+ \theta \bar{\l}\right) |S_\a w|^2 + (A_{n,p}-\theta) \bar{\l} \sum_{\a=1}^N |S_\a w|^2 \\
    &\geq& 0,
\end{eqnarray*}
where we have used $\mathring{R} \in \mathcal{C}\left(\frac{n+2}{2}, \theta \right)$ and $\theta < A_{n,p}$. Note that $\mathring{R} \in \mathcal{C}\left(\frac{n+2}{2}, \theta \right)$ implies that either  $\mathring{R} \in \mathring{\mathcal{C}}\left(\frac{n+2}{2}, A_{n,p} \right)$ or $\bar{\l}=0$. 
If there exist $p\in M$ such that $\bar{\l}(p) >0$, then $\mathring{R}_p \in \mathring{\mathcal{C}}\left(\frac{n+2}{2}, A_{n,p} \right)$. Then $\sum_{\a=1}^n  \left(\l_\a+ \theta \bar{\l}\right) |S_\a w|^2 =0$ at $p$ implies $\omega(p) = 0$. Since $|\omega|$ is a constant, we conclude that $\omega \equiv 0$ on $M$. 
Otherwise, $(M,g)$ is scalar flat and hence flat by Proposition \ref{prop basic cones}.

Therefore, either $(M,g)$ is flat or $b_p(M,\R)=0$.
\end{proof}

Finally, we prove the homological sphere theorem. 

\begin{proof}[Proof of Theorem \ref{thm homology sphere}]
Note that for $2\leq p \leq \frac{n}{2}$, $A_{n,p}$ increases as $p$ increases. Therefore, 
$$A_{n,p} \geq A_{n,2} = \frac{2(n-1)(3n-4)}{3n^3-12n^2+14n-8} > \frac{2}{n+2},$$
where the last inequality holds for any $n \geq 3$. 
If $\mathring{R}\in \mathcal{C}\left(\frac{n+2}{2}, \frac{2}{n+2} \right)$, then $\mathring{R}\in \mathcal{C}\left(\frac{n+2}{2}, \theta \right)$ with $\theta < A_{n,p}$. By Theorem \ref{thm Betti}, we have that either $(M,g)$ is flat or $b_p(M,\R)=0$ for all $2\leq p \leq \frac{n}{2}$. 
    
By Proposition \ref{prop Ricci}, $\mathring{R}\in \mathcal{C}\left(\frac{n+2}{2}, \theta \right)$ for some $\theta < \frac{n+2}{2}$ implies $\Ric \geq \delta \bar{\l} g \geq 0$, where $\delta=\frac{n-1}{n+4}(2-(n+2)\theta)>0$. If there exists $p\in M$ such that $\bar{\l}(p)>0$, then the Ricci curvature is positive at $p$ and we conclude $b_1(M,\R)=0$ (see \cite[Theorem 3.5]{Libook}). Thus, all the Betti numbers vanish and $(M,g)$ is a rational homology sphere.
Otherwise, $\bar{\l}\equiv 0$ on $M$ and $(M,g)$ is flat by Proposition \ref{prop basic cones}. 
Hence, $(M,g)$ is either flat or a rational homology sphere. 
\end{proof}

\section{K\"ahler Manifolds}

In this section, we prove Theorem \ref{thm Kahler flat}. We recall some observations and identities needed in the proof and refer the reader to \cite{Li22Kahler} for a more detailed account. 

Let $(V,g,J)$ be a complex Euclidean vector space of complex dimension $m\geq 2$, where $J$ is the complex structure. 
As observed in \cite{Li22Kahler}, $S^2_0(V)$ admits the decomposition
\begin{equation*}
    S^2_0(V) = E^+ \oplus E^{-},
\end{equation*}
where 
$$E^+=\spn\{u \odot v -Ju \odot Jv : u, v \in V \}$$ and $E^-=(E^+)^\perp$. It was shown in \cite{BK78} that, on $(\mathbb{CP}^m, g_{FS})$, $E^+$ is the eigenspace associated with the eigenvalue $4$ of the curvature operator of the second kind while $E^-$ is the eigenspace associated with the eigenvalue $-2$. Note that $\dim(E^-)=m^2-1$ and $\dim(E^+)=m(m+1)$. 

As constructed in \cite{Li22Kahler}, if $\{e_1, \cdots, e_m, Je_1, \cdots, Je_m\}$ is an orthonormal basis of $V$, then $E^-$ has an orthonormal basis given by
$$\{\vp^-_{ij}\}_{1\leq i < j \leq m} \cup \{\psi^-_{ij}\}_{1\leq i < j \leq m} \cup \{\eta_k\}_{k=1}^{m-1},$$
where
\begin{eqnarray*}
\vp^{-}_{ij}  &=& \tfrac{1}{2} \left( e_i \odot e_j + Je_i \odot Je_j \right), \text{ for } 1 \leq i < j \leq m, \\
\psi^{-}_{ij}  &=& \tfrac{1}{2} \left( e_i \odot J e_j - Je_i \odot e_j \right), \text{ for } 1 \leq i < j \leq m,\\
\eta_k &=& \tfrac{k}{\sqrt{8k(k+1)}}  (e_{k+1}\odot e_{k+1} +Je_{k+1} \odot Je_{k+1}) \\
    &&  - \tfrac{1}{\sqrt{8k(k+1)}}\sum_{i=1}^{k}(e_i \odot e_i +Je_i \odot Je_i), \\
    && \text{ for } 1\leq k \leq m-1, 
\end{eqnarray*} 
and $E^+$ has an orthonormal basis given by 
$$\{\vp^{+}_{ij}\}_{1\leq i < j \leq m} \cup\{\psi^{+}_{ij}\}_{1\leq i < j \leq m} \cup \{\theta_i \}_{i=1}^{2m},$$ where
\begin{eqnarray*}
\vp^+_{ij}  &=& \tfrac{1}{2} \left( e_i \odot e_j - Je_i \odot Je_j \right), \text{ for } 1 \leq i < j \leq m, \\
\psi^{+}_{ij}  &=& \tfrac{1}{2} \left( e_i \odot J e_j + Je_i \odot e_j \right), \text{ for } 1 \leq i < j \leq m, \\
    \theta_{i} &=& \tfrac{1}{2\sqrt{2}} \left( e_i \odot e_i -Je_i \odot Je_i \right), \text{ for } 1\leq i \leq m, \\
     \theta_{m+i} &=& \tfrac{1}{\sqrt{2}} e_i \odot J e_i, \text{ for } 1\leq i \leq m.
\end{eqnarray*}
It was calculated in \cite{Li22Kahler} that
\begin{eqnarray}\label{eq 8.10}
\mathring{R}(\vp^-_{ij}, \vp^-_{ij}) +\mathring{R}(\psi^-_{ij}, \psi^-_{ij})  =  -2 R(e_i, Je_i, e_j, Je_j),
\end{eqnarray}
for  $1\leq i < j \leq m$, and
\begin{equation}\label{eq 8.13}
\mathring{R}(\theta_{i}, \theta_{i}) = \mathring{R}(\theta_{m+i}, \theta_{m+i}) =R(e_i,Je_i, e_i, Je_i), 
\end{equation}
for $1\leq i \leq m$.
Moreover, we have (see \cite[Lemma 4.3 and Lemma 4.4]{Li22Kahler}) that
\begin{equation}\label{eq 8.5}
2m(2m-1)\bar{\l}=\sum_{1\leq i< j \leq m} \left( \mathring{R}(\vp^+_{ij}, \vp^+_{ij}) + \mathring{R}(\psi^+_{ij}, \psi^+_{ij}) \right) + \sum_{i=1}^{2m} \mathring{R}(\theta_{i}, \theta_{i})   
\end{equation}
and
\begin{equation}\label{eq 8.11}
 -(m-1)(2m-1)\bar{\l}=   \sum_{1\leq i <j \leq m} \left( \mathring{R}(\vp^-_{ij}, \vp^-_{ij}) + \mathring{R}(\psi^-_{ij}, \psi^-_{ij}) \right) +\sum_{k=1}^{m-1} \mathring{R}(\eta_k,\eta_k). 
\end{equation}

We are ready to prove Theorem \ref{thm Kahler flat}. 
\begin{proof}[Proof of Theorem \ref{thm Kahler flat}]
(1). We will only prove the statement for $\mathring{R} \in \mathcal{C}(\a, B_{m,\a})$, as the case $-\mathring{R} \in \mathcal{C}(\a, B_{m,\a})$ differs only by flipping signs. 
Fix $p\in M$. Let $\{e_1, \cdots, e_m, Je_1, \cdots, Je_m\}$ be an orthonormal basis of $V=T_pM$.  

\textbf{Case 1:} $1\leq \a < m^2-1$. The assumption $\mathring{R} \in \mathcal{C}(\a, B_{m,\a})$ implies
\begin{eqnarray}\label{eq 8.2}
&& -\a B_{m,\a} \bar{\l} \\ \nonumber
    &\leq&   \l_1  + \cdots +\l_{\a} \\ \nonumber 
    &\leq & \frac{\a}{m^2-1} \left( \l_1 +  \cdots +\l_{m^2-1}\right) \\ \nonumber
    &\leq & \frac{\a}{m^2-1} \left(\sum_{1\leq i <j \leq m} \left( \mathring{R}(\vp^-_{ij}, \vp^-_{ij}) + \mathring{R}(\psi^-_{ij}, \psi^-_{ij}) \right) +\sum_{k=1}^{m-1} \mathring{R}(\eta_k,\eta_k) \right) \\ \nonumber
    &= & -\a B_{m,\a} \bar{\l}, 
\end{eqnarray}
where we have used \eqref{eq 8.11} in the last step.

Therefore, we must have equality in each inequality of \eqref{eq 8.2}. It follows that
\begin{equation*}
    \l_1 = \cdots =\l_{m^2-1} = -\frac{2m-1}{m+1} \bar{\l}. 
\end{equation*}
and $E^-$ is a subspace of the eigenspace associated with the eigenvalue $-\frac{2m-1}{m+1} \bar{\l}$ of $\mathring{R}$. 
Below we will show that this information is sufficient to conclude constant holomorphic sectional curvature at $p$. 

By \eqref{eq 8.10}, we have for any $1\leq i < j \leq m$,
\begin{equation}\label{eq 8.3}
    R(e_i, Je_i, e_j, Je_j) = \frac{2m-1}{m+1} \bar{\l}.
\end{equation}
Note that for any $1\leq i < j \leq m$, 
$$\xi :=\frac{1}{4} \left(e_i \odot e_i +Je_i\odot Je_i -e_j \odot e_j -Je_j \odot Je_j \right)$$
is a traceless symmetric two-tensor in $E^-$ with $|\xi|=1$. Therefore, we have 
\begin{equation}\label{eq 8.45}
 \mathring{R}(\xi,\xi)=-\frac{2m-1}{m+1} \bar{\l}.   
\end{equation}
A straightforward calculation shows 
\begin{equation}\label{eq 8.4}
     \mathring{R}(\xi,\xi) = -\frac{1}{2} R(e_i,Je_i,e_i,Je_i) -\frac{1}{2}R(e_j,Je_j,e_j,Je_j) +R(e_i,Je_i,e_j, Je_j), 
\end{equation}
Combining \eqref{eq 8.3}, \eqref{eq 8.45}, and \eqref{eq 8.4}, we arrive at 
\begin{equation*}
    R(e_i, Je_i, e_i, Je_i) + R(e_j,Je_j, e_j, Je_j)  =4 \frac{2m-1}{m+1} \bar{\l}
\end{equation*}
for any $1\leq i < j \leq m$. It follows that 
\begin{equation*}
R(e_i, Je_i, e_i, Je_i) =2 \frac{2m-1}{m+1} \bar{\l}
\end{equation*}
for all $1\leq i \leq m$. Since the orthonormal basis is arbitrary, this shows that $M$ has constant nonnegative holomorphic sectional curvature at $p$.

\textbf{Case 2:} $m^2-1 < \a < (2m-1)(m+1)$. 
Let $A$ be the collection of the values $\mathring{R}(\theta_i,\theta_i)$ for $1\leq i \leq 2m$, and $\mathring{R}(\vp^{+}_{ij},\vp^{+}_{ij})$ and $\mathring{R}(\psi^{+}_{ij},\psi^{+}_{ij})$ for $1\leq i < j \leq m$. 
By \eqref{eq 8.5}, $\bar{a}$, the average of all values in $A$, is given by
\begin{equation*}
    \bar{a}=\frac{2(2m-1)}{m+1}\bar{\l}. 
\end{equation*}
By Lemma \ref{lemma average}, we have 
\begin{equation}\label{eq 8.21}
    f(A, (\a-m^2+1)) \leq (\a-m^2+1)) \bar{a},
\end{equation}
where $f$ is the function defined in Lemma \ref{lemma average}. 
The condition $\mathring{R} \in \mathcal{C}(\a, B_{m,\a})$ implies
\begin{eqnarray}\label{eq 8.6}
&& -\a B_{m,\a} \bar{\l} \\ \nonumber
&\leq &  \l_1 + \cdots + \l_{\a}  \\ \nonumber 
& \leq & \sum_{1\leq i <j \leq m} \left( \mathring{R}(\vp^-_{ij}, \vp^-_{ij}) + \mathring{R}(\psi^-_{ij}, \psi^-_{ij}) \right) +\sum_{k=1}^{m-1} \mathring{R}(\eta_k,\eta_k) \\ \nonumber
&& +f(A, (\a-m^2+1)) \\ \nonumber
&\leq & -(m-1)(2m-1)\bar{\l}  +(\a-m^2+1)) \bar{a} \\ \nonumber
&=& -\a B_{m,\a} \bar{\l},
\end{eqnarray}
where we have used \eqref{eq 8.11} and \eqref{eq 8.21}. It follows that we must have equality in each inequality of \eqref{eq 8.6}. In particular, we have equality in \eqref{eq 8.21}, which happens only when all the values in $A$ are equal to $\bar{a}$. By \eqref{eq 8.13}, we have 
\begin{equation*}\label{eq R positive basis alpha}
    R(e_i,Je_i, e_i, Je_i) = \mathring{R}(\theta_{i}, \theta_{i}) = 2\frac{2m-1}{m+1}\bar{\l}
\end{equation*}
for all $1\leq i \leq m$. Since the orthonormal basis is arbitrary, we have proved that $(M,g)$ has constant nonnegative holomorphic sectional curvature at $p$. 

Finally, the conclusion that $(M,g)$ has constant nonnegative holomorphic sectional curvature follows from Schur's lemma for K\"ahler manifolds (see for instance \cite[Theorem 7.5]{KN69}).

(2). If we assume $\mathring{R} \in \mathcal{C}(\a, \theta)$, then  \eqref{eq 8.2} and \eqref{eq 8.6} in the proof of part (1) would become
\begin{equation*}
    -\a \theta \bar{\l} \leq -\a B_{m,\a} \bar{\l}.
\end{equation*}
This works for $\a=m^2-1$ as well. 

If $\theta < B_{m,\a}$, then $\bar{\l}=0$. By Proposition \ref{prop basic cones}, $(M,g)$ must be flat. 

\end{proof}

\section*{Acknowledgments}
The author would like to thank Prof. Matt Gursky for helpful discussions that initiated the consideration of the curvature cones $\mathcal{C}(\a,\theta)$ in this paper. He also thanks the referee for valuable feedback which improved the readability of the paper.

\bibliographystyle{alpha}
\bibliography{ref}

\begin{thebibliography}{NPWW23}

\bibitem[BE69]{BE69}
M.~Berger and D.~Ebin.
\newblock Some decompositions of the space of symmetric tensors on a {R}iemannian manifold.
\newblock {\em J. Differential Geometry}, 3:379--392, 1969.

\bibitem[Ber60]{Berger60}
M.~Berger.
\newblock Les vari\'et\'es {R}iemanniennes {$(1/4)$}-pinc\'ees.
\newblock {\em Ann. Scuola Norm. Sup. Pisa Cl. Sci. (3)}, 14:161--170, 1960.

\bibitem[Bes08]{Besse08}
Arthur~L. Besse.
\newblock {\em Einstein manifolds}.
\newblock Classics in Mathematics. Springer-Verlag, Berlin, 2008.
\newblock Reprint of the 1987 edition.

\bibitem[BK78]{BK78}
Jean-Pierre Bourguignon and Hermann Karcher.
\newblock Curvature operators: pinching estimates and geometric examples.
\newblock {\em Ann. Sci. \'{E}cole Norm. Sup. (4)}, 11(1):71--92, 1978.

\bibitem[Bor60]{Borel60}
Armand Borel.
\newblock On the curvature tensor of the {H}ermitian symmetric manifolds.
\newblock {\em Ann. of Math. (2)}, 71:508--521, 1960.

\bibitem[Bre08]{Brendle08}
Simon Brendle.
\newblock A general convergence result for the {R}icci flow in higher dimensions.
\newblock {\em Duke Math. J.}, 145(3):585--601, 2008.

\bibitem[Bre10]{Brendle10book}
Simon Brendle.
\newblock {\em Ricci flow and the sphere theorem}, volume 111 of {\em Graduate Studies in Mathematics}.
\newblock American Mathematical Society, Providence, RI, 2010.

\bibitem[BS08]{BS08}
Simon Brendle and Richard~M. Schoen.
\newblock Classification of manifolds with weakly {$1/4$}-pinched curvatures.
\newblock {\em Acta Math.}, 200(1):1--13, 2008.

\bibitem[BS09]{BS09}
Simon Brendle and Richard Schoen.
\newblock Manifolds with {$1/4$}-pinched curvature are space forms.
\newblock {\em J. Amer. Math. Soc.}, 22(1):287--307, 2009.

\bibitem[BS11]{BS11}
Simon Brendle and Richard Schoen.
\newblock Curvature, sphere theorems, and the {R}icci flow.
\newblock {\em Bull. Amer. Math. Soc. (N.S.)}, 48(1):1--32, 2011.

\bibitem[BW08]{BW08}
Christoph B{\"o}hm and Burkhard Wilking.
\newblock Manifolds with positive curvature operators are space forms.
\newblock {\em Ann. of Math. (2)}, 167(3):1079--1097, 2008.

\bibitem[CGT23]{CGT23}
Xiaodong Cao, Matthew Gursky, and Hung Tran.
\newblock Curvature of the second kind and a conjecture of {N}ishikawa.
\newblock {\em Comment. Math. Helv.}, 98(1):195--216, 2023.

\bibitem[Che91]{Chen91}
Haiwen Chen.
\newblock Pointwise {$\frac14$}-pinched {$4$}-manifolds.
\newblock {\em Ann. Global Anal. Geom.}, 9(2):161--176, 1991.

\bibitem[CTZ12]{CTZ12}
Bing-Long Chen, Siu-Hung Tang, and Xi-Ping Zhu.
\newblock Complete classification of compact four-manifolds with positive isotropic curvature.
\newblock {\em J. Differential Geom.}, 91(1):41--80, 2012.

\bibitem[CV60]{CV60}
Eugenio Calabi and Edoardo Vesentini.
\newblock On compact, locally symmetric {K}\"{a}hler manifolds.
\newblock {\em Ann. of Math. (2)}, 71:472--507, 1960.

\bibitem[CZ06]{CZ06}
Bing-Long Chen and Xi-Ping Zhu.
\newblock Ricci flow with surgery on four-manifolds with positive isotropic curvature.
\newblock {\em J. Differential Geom.}, 74(2):177--264, 2006.

\bibitem[DF24]{DF24}
Zhi-Lin Dai and Hai-Ping Fu.
\newblock Einstein manifolds and curvature operator of the second kind.
\newblock {\em Calc. Var. Partial Differential Equations}, 63(2):Paper No. 53, 22, 2024.

\bibitem[DFY24]{DFY24}
Zhi-Lin Dai, Hai-Ping Fu, and Deng-Yun Yang.
\newblock Manifolds with harmonic {W}eyl tensor and nonnegative curvature operator of the second kind.
\newblock {\em J. Geom. Phys.}, 195:Paper No. 105040, 2024.

\bibitem[FL24]{FL24}
Harry Fluck and Xiaolong Li.
\newblock The curvature operator of the second kind in dimension three.
\newblock {\em J. Geom. Anal.}, 34(6):Paper No. 187, 19, 2024.

\bibitem[GM75]{GM75}
S.~Gallot and D.~Meyer.
\newblock Op\'{e}rateur de courbure et laplacien des formes diff\'{e}rentielles d'une vari\'{e}t\'{e} riemannienne.
\newblock {\em J. Math. Pures Appl. (9)}, 54(3):259--284, 1975.

\bibitem[Ham82]{Hamilton82}
Richard~S. Hamilton.
\newblock Three-manifolds with positive {R}icci curvature.
\newblock {\em J. Differential Geom.}, 17(2):255--306, 1982.

\bibitem[Ham86]{Hamilton86}
Richard~S. Hamilton.
\newblock Four-manifolds with positive curvature operator.
\newblock {\em J. Differential Geom.}, 24(2):153--179, 1986.

\bibitem[Ham97]{Hamilton97}
Richard~S. Hamilton.
\newblock Four-manifolds with positive isotropic curvature.
\newblock {\em Comm. Anal. Geom.}, 5(1):1--92, 1997.

\bibitem[Kas93]{Kashiwada93}
Toyoko Kashiwada.
\newblock On the curvature operator of the second kind.
\newblock {\em Natur. Sci. Rep. Ochanomizu Univ.}, 44(2):69--73, 1993.

\bibitem[Kli61]{Klingenberg61}
Wilhelm Klingenberg.
\newblock \"uber {R}iemannsche {M}annigfaltigkeiten mit positiver {K}r\"ummung.
\newblock {\em Comment. Math. Helv.}, 35:47--54, 1961.

\bibitem[KN69]{KN69}
Shoshichi Kobayashi and Katsumi Nomizu.
\newblock {\em Foundations of differential geometry. {V}ol. {II}}.
\newblock Interscience Tracts in Pure and Applied Mathematics, No. 15 Vol. II. Interscience Publishers John Wiley \& Sons, Inc., New York-London-Sydney, 1969.

\bibitem[Koi79a]{Koiso79a}
Norihito Koiso.
\newblock A decomposition of the space {${\mathcal{M}}$} of {R}iemannian metrics on a manifold.
\newblock {\em Osaka Math. J.}, 16(2):423--429, 1979.

\bibitem[Koi79b]{Koiso79b}
Norihito Koiso.
\newblock On the second derivative of the total scalar curvature.
\newblock {\em Osaka Math. J.}, 16(2):413--421, 1979.

\bibitem[Li12]{Libook}
Peter Li.
\newblock {\em Geometric analysis}, volume 134 of {\em Cambridge Studies in Advanced Mathematics}.
\newblock Cambridge University Press, Cambridge, 2012.

\bibitem[Li22]{Li22JGA}
Xiaolong Li.
\newblock Manifolds with {$4\frac 12$}-{P}ositive {C}urvature {O}perator of the {S}econd {K}ind.
\newblock {\em J. Geom. Anal.}, 32(11):281, 2022.

\bibitem[Li23a]{Li22Kahler}
Xiaolong Li.
\newblock K\"{a}hler manifolds and the curvature operator of the second kind.
\newblock {\em Math. Z.}, 303(4):101, 2023.

\bibitem[Li23b]{Li22PAMS}
Xiaolong Li.
\newblock K\"{a}hler surfaces with six-positive curvature operator of the second kind.
\newblock {\em Proc. Amer. Math. Soc.}, 151(11):4909--4922, 2023.

\bibitem[Li24a]{Li21}
Xiaolong Li.
\newblock Manifolds with nonnegative curvature operator of the second kind.
\newblock {\em Commun. Contemp. Math.}, 26(3):Paper No. 2350003, 2024.

\bibitem[Li24b]{Li22product}
Xiaolong Li.
\newblock Product manifolds and the curvature operator of the second kind.
\newblock {\em Pacific J. Math.}, 332(1):167--193, 2024.

\bibitem[Liu13]{Liu13}
Gang Liu.
\newblock 3-manifolds with nonnegative {R}icci curvature.
\newblock {\em Invent. Math.}, 193(2):367--375, 2013.

\bibitem[Mey71]{Meyer71}
Daniel Meyer.
\newblock Sur les vari\'{e}t\'{e}s riemanniennes \`a op\'{e}rateur de courbure positif.
\newblock {\em C. R. Acad. Sci. Paris S\'{e}r. A-B}, 272:A482--A485, 1971.

\bibitem[MRS20]{MRS20}
Josef Mike\v{s}, Vladimir Rovenski, and Sergey~E. Stepanov.
\newblock An example of {L}ichnerowicz-type {L}aplacian.
\newblock {\em Ann. Global Anal. Geom.}, 58(1):19--34, 2020.

\bibitem[Nis86]{Nishikawa86}
Seiki Nishikawa.
\newblock On deformation of {R}iemannian metrics and manifolds with positive curvature operator.
\newblock In {\em Curvature and topology of {R}iemannian manifolds ({K}atata, 1985)}, volume 1201 of {\em Lecture Notes in Math.}, pages 202--211. Springer, Berlin, 1986.

\bibitem[NPW23]{NPW22}
Jan Nienhaus, Peter Petersen, and Matthias Wink.
\newblock Betti numbers and the curvature operator of the second kind.
\newblock {\em J. Lond. Math. Soc. (2)}, 108(4):1642--1668, 2023.

\bibitem[NPWW23]{NPWW22}
Jan Nienhaus, Peter Petersen, Matthias Wink, and William Wylie.
\newblock Holonomy restrictions from the curvature operator of the second kind.
\newblock {\em Differential Geometry and its Applications}, 88:102010, 2023.

\bibitem[NW07]{NW07}
Lei Ni and Baoqiang Wu.
\newblock Complete manifolds with nonnegative curvature operator.
\newblock {\em Proc. Amer. Math. Soc.}, 135(9):3021--3028, 2007.

\bibitem[NW10]{NW10}
Lei Ni and Burkhard Wilking.
\newblock Manifolds with {$1/4$}-pinched flag curvature.
\newblock {\em Geom. Funct. Anal.}, 20(2):571--591, 2010.

\bibitem[OT79]{OT79}
Koichi Ogiue and Shun-ichi Tachibana.
\newblock Les vari\'{e}t\'{e}s riemanniennes dont l'op\'{e}rateur de courbure restreint est positif sont des sph\`eres d'homologie r\'{e}elle.
\newblock {\em C. R. Acad. Sci. Paris S\'{e}r. A-B}, 289(1):A29--A30, 1979.

\bibitem[Pet16]{Petersen2016book}
Peter Petersen.
\newblock {\em Riemannian geometry}, volume 171 of {\em Graduate Texts in Mathematics}.
\newblock Springer, Cham, third edition, 2016.

\bibitem[PW21]{PW21}
Peter Petersen and Matthias Wink.
\newblock New curvature conditions for the {B}ochner {T}echnique.
\newblock {\em Invent. Math.}, 224(1):33--54, 2021.

\bibitem[Tac74]{Tachibana74}
Shun-ichi Tachibana.
\newblock A theorem of {R}iemannian manifolds of positive curvature operator.
\newblock {\em Proc. Japan Acad.}, 50:301--302, 1974.

\end{thebibliography}

\end{document}